\documentclass{ws-m3as}
\usepackage{mathrsfs}
\usepackage{url,color,epsfig}

\usepackage{graphicx,subfigure}
\usepackage{stfloats}
\usepackage{colortbl}
\usepackage{amsmath}
\usepackage[misc]{ifsym}
\usepackage{latexsym,amsfonts,amssymb}
\usepackage{float}
\begin{document}

\markboth{H. Zhang et al.}{Convergence analysis for time-fractional nonlinear subdiffusion equations }

%
\catchline{}{}{}{}{}
%

\title{Convergence analysis of the time-stepping numerical methods for time-fractional nonlinear subdiffusion equations }

\author{Hui Zhang}

\address{School of Mathematics, Shandong University, \\
 Jinan 250100, PR China
\\
zhangh@sdu.edu.cn}

\author{Fanhai Zeng \footnote{Corresponding author.}
}

\address{School of Mathematics, Shandong University, \\
 Jinan 250100, PR China \\
fanhai\_zeng@sdu.edu.cn}

\author{Xiaoyun Jiang}

\address{School of Mathematics, Shandong University, \\
 Jinan 250100, PR China \\
wqjxyf@sdu.edu.cn}

\author{George Em Karniadakis}

\address{Division of Applied Mathematics and Engineering, Brown University, \\
Providence RI, 02912\\
george\_karniadakis@brown.edu}

\maketitle

\begin{history}
\received{(Day Month Year)}
\revised{(Day Month Year)}
\comby{(xxxxxxxxxx)}
\end{history}

\begin{abstract}
In 1986, Dixon and McKee
developed a discrete fractional Gr\"{o}nwall inequality
[Z. Angew. Math. Mech., 66 (1986), pp. 535--544],
which can be seen as a generalization of the classical discrete  Gr\"{o}nwall inequality.
However, this generalized discrete Gr\"{o}nwall inequality  and its variant
[SIAM J. Numer. Anal., 57 (2019), pp. 1524--1544] have not been widely
applied  in the numerical analysis of the time-stepping methods
for the time-fractional evolution equations.
The main purpose of this paper is to show how to apply the generalized discrete
Gr\"{o}nwall inequality to prove the convergence of a class of time-stepping
numerical methods
for time-fractional nonlinear subdiffusion equations,  including the popular
fractional backward difference type methods of order one and two,
and the  fractional Crank-Nicolson type methods.
We obtain the optimal $L^2$ error estimate in space discretization for multi-dimensional problems.
The convergence of the fast time-stepping numerical methods is also proved in a simple manner.
The present work  unifies the convergence analysis of several existing  time-stepping schemes.
Numerical examples are provided to verify the effectiveness of the present method.
\end{abstract}


\section{Introduction}\label{intro}
The aim of this paper is to analyze the convergence of the time-stepping  numerical schemes
for the following time-fractional nonlinear  subdiffusion equation with a reaction term $f(u)$:
\begin{equation}\label{e1.1}\left\{\begin{aligned}
&{}_{0}^{C} D_{t}^{\alpha}u=\Delta u+f(u), &&  \text{ in } \Omega\times (0,T],T>0,\\
&u=u_0,&&  \text{ in } \bar{\Omega},\\
&u=0, &&  \text{ on } \partial\Omega\times [0,T],
\end{aligned}\right.\end{equation}
where   $\Omega$ is a convex domain in $\mathbb{R}^d$ with a smooth boundary, $\Delta$ is the Laplace operator
defined on $\Omega$ with a homogenous boundary condition,
and $_{0}^{C}D_{t}^{\alpha}u$ is the Caputo  fractional derivative of order $0<\alpha<1$, which is  defined by
\begin{equation}\begin{aligned}\label{e1.133}
{}_{0}^{C}D_{t}^{\alpha}u(t)=\frac{1}{\Gamma(1-\alpha)}
\int_{0}^{t}u'(s)(t-s)^{-\alpha}\mathrm{d}s.
\end{aligned}\end{equation}
We employ the Galerkin finite element method (FEM) in space approximation.
The spatial approximation can also be performed  by other methods, for example, if  $\Omega$ is regular, then
finite difference methods or spectral methods can be applied.

The non-locality of the fractional derivative operator \eqref{e1.133} causes
a lot of difficulty for solving \eqref{e1.1}.
Generally speaking, the approximation
of  ${}_{0}^{C}D_{t}^{\alpha}u(t)$ at $t=t_n$  can be written as
\begin{equation}\label{eq:wnj}
\sum_{k=0}^nw_{n,k}u^k,\quad 0\leq k \leq n,\ 0<n\leq n_T,
\end{equation}
where the coefficients $w_{n,k}$ are determined by the specific numerical method
for the approximation of the fractional operator.\cite{Alikhanov15,Hongwei2017,LinXu07,Lub86,SunWu06,YanKF18,ZhuXu2019}
Direct computation of \eqref{eq:wnj} is costly,
requiring $O(n_T)$ active memory
and $O(n_T^2)$  operations. The computational difficulty
can be resolved by developing fast memory-saving algorithms.\cite{BafHes17b,BanjaiLopez18,ChenZZCW19,GuoZeng19,JiangZZZ17,JingLi10,LopLubSch08,SunNieDeng19,ZengTBK2018,ZhuXu2019}
The non-locality of fractional operators also makes the
numerical analysis of fractional partial differential equations (PDEs)
much more complicated than that of local PDEs.  As is well known,
the discrete Gr\"{o}nwall  inequality (see Lemma \ref{le3.1} with $\alpha\to 1$)
provides a powerful tool to analyze the stability and convergence
of the numerical methods for  integer-order PDEs.
How to develop and use the discrete fractional Gr\"{o}nwall type inequalities
to analyze the numerical methods for   fractional PDEs
has been reported much less and this is the topic of this current work.



The discrete fractional Gr\"{o}nwall type inequalities based on the specific time-stepping
methods have been established by some researchers.\cite{JinLiZhou18b,LiaoLZ18,LiaoWZ19,YangZeng2019}
Jin et al.\cite{JinLiZhou18b}
established a fractional version of the discrete Gr\"{o}nwall type inequality
based on the convolution quadrature generated by the fractional backward difference formula
of order $p$ (FBDF-$p$) and the L1 formula. In Refs. \refcite{LiaoLZ18} and \refcite{LiaoWZ19}, the authors developed the
discrete fractional Gr\"{o}nwall type inequalities based on the interpolation method,
such as the L1 method generated by linear interpolation\cite{LinXu07,StynesOG17,SunWu06}
and the  Alikhanov formula generated by quadratic interpolation.\cite{Alikhanov15}
These Gr\"{o}nwall type inequalities have been applied to analyze the convergence of
numerical methods for a variety of  nonlinear fractional
PDEs.\cite{DuYang2020,JinLiZhou18b,LiLiaoSWZ18,LiZhangZhang18,LiaoYZ19}

In addition to the aforementioned  discrete fractional Gr\"{o}nwall type inequalities,
there exists  a generalized  discrete Gr\"{o}nwall  inequality (see Lemma \ref{le3.1}) proposed in 1986  by Dixon and McKee (see Ref. \refcite{DixonMcKee86}), which can be seen as a generalization of the classical discrete Gr\"{o}nwall  inequality and is independent of specific time-stepping methods.
The generalized  discrete Gr\"{o}nwall  inequality and its variants  have been  widely
applied to analyze the  convergence of the numerical methods for the fractional ordinary differential equations and the integral equations with weakly singular kernels.\cite{BrunnerTang89,CaoXu13,LiYiChen16,ZengZK17} To the best of the authors' knowledge, this generalized  discrete Gr\"{o}nwall inequality has not been widely applied to analyze the convergence of time-stepping numerical methods for the time-fractional PDEs except for
some limited works.\cite{AL-MaskariKaraa19,GonPal99,LeMcLMus16} The goal of this work is to show how to apply the generalized discrete Gr\"{o}nwall inequality to prove the convergence of a class of time-stepping numerical methods for time-fractional nonlinear PDEs of the form \eqref{e1.1}.

The main contributions of this work are listed below:
\begin{itemlist}
\item The generalized discrete Gr\"{o}nwall's inequality is applied
to prove the convergence of a class of fully implicit time-stepping Galerkin
FEMs for \eqref{e1.1}, where the time direction  is approximated
by the convolution quadrature with correction terms.
The use of the generalized discrete Gr\"{o}nwall  inequality in this paper is very simple and straightforward; see Section \ref{sec-4}.
\item The convergence of the fast time-stepping Galerkin FEMs for
  \eqref{e1.1} is proved.  Our proof is based on the convergence
of the direct computational method,
which is   simpler than that of the existing fast methods; see Ref. \refcite{SunNieDeng19}.
\end{itemlist}

To the best of   authors' knowledge,
this is the first work that unifies the convergence analysis of the
popular (fast) time-stepping numerical schemes  for solving  \eqref{e1.1},
including the fractional backward difference type methods of order one and two,\cite{Lub86,TianZD15} the fractional Crank--Nicolson
type methods,\cite{JinLiZhou18c,ZengLLT15}
and the recently developed BN-$\theta$ method,\cite{Yinbaoli2020}
see Section \ref{sec-5}.

The convolution quadrature with correction terms has been widely  applied to resolve
the initial singularity of the time-fractional PDEs.\cite{CueLubPal06,JinLiZhou18b,WangZhou20,Yinbaoli2020}
However, the convergence analysis of time-stepping schemes with correction terms
is limited; the current paper presents an approach to analyze the convergence of this kind
time-stepping numerical methods.
The present convolution quadrature   with correction terms
is different from the ones in Refs. \refcite{JinLiZhou17} and \refcite{YanKF18},
 where the first several steps of the schemes are  corrected.

The main difference of the present work from the previous ones \cite{JinLiZhou18b,LiaoLZ18,LiaoWZ19,YangZeng2019} is that
we adopt the generalized discrete Gr\"{o}nwall inequality
to prove the convergence of the numerical methods.
Our analysis is simple and straightforward, and can be extended to
analyze the numerical methods for a broader class of time-fractional
evolution equations.
\section{The numerical schemes}\label{sec-3}
\subsection{Discretization of the Caputo fractional derivative}
The interval $[0,T]$ is divided into $n_T\in \mathbb{N}$  subintervals
with a time step size $\tau=T/n_T$ and
grid points $t_n=n\tau,0\leq n \leq n_T$. Denote by $u^n=u^n(\cdot)=u(\cdot,t_n)$
for notational simplicity.

Assume that the solution $u$ of \eqref{e1.1} satisfies
\begin{equation} \label{solution-u}
u(t)-u(0)=\sum_{k=1}^m\hat{u}_kt^{\delta_k} + \tilde{u}(t)t^{\delta_{m+1}},
\quad 0\leq t\leq T,
\end{equation}
where $0<\delta_1<\cdots<\delta_m<\delta_{m+1}$ and $\tilde{u}(t)\in L^2([0,T];X)$.
The assumption \eqref{solution-u} is used in obtaining the truncation error in time
discretization, which holds for the linear equation of the form \eqref{e1.1}.
For example, if $f=u$, then $\delta_k=k\alpha$; see Ref. \refcite[Theorem 5]{Luchko12}.
If $f=g(\cdot,t)$, $g$ is sufficient smooth in time, then
$\delta_k \in \{\delta_{\ell,j}|\delta_{\ell,j}=\ell+j\alpha,\ell\in \mathbb{Z}^+,j\in \mathbb{N}\}$;
see Refs. \refcite{CueLubPal06} and \refcite{Luchko12}.
For  the time-fractional Allen--Cahn equation, i.e., $f=u(1-u^2)$,
one has $\delta_1=\alpha$; see Ref. \refcite{WangZhou20}.

%

The following lemma is a reformulation of Lemma 3.5 in Ref. \refcite{Lub86}, which
is useful in the construction of the numerical method for
the Caputo fractional operator.
\begin{lemma}[see Ref. \refcite{Lub86}]\label{lem2-1}
Let $u(t)=t^{\gamma},\gamma>-1$ and $0\leq  \alpha \leq 1$. Then
$${}_{0}^{RL}D_{t}^{\alpha}u(t)|_{t=t_n} = \tau^{-\alpha}\sum_{k=1}^n\omega^{(\alpha)}_{n-k}u(t_k)
+ O(\tau^pt_{n}^{\gamma-p-\alpha}) + O(\tau^{\gamma+1}t_{n}^{-\alpha-1}),$$
where ${}_{0}^{RL}D_{t}^{\alpha}$ is the Riemann--Liouville   fractional derivative operator defined by
$$ {}_{0}^{RL}D_{t}^{\alpha}u(t)
= \frac{1}{\Gamma(1-\alpha)} \frac{d}{dt}\int_0^{t}(t-s)^{-\alpha}u(s)ds,$$
the convolution weights $\omega^{(\alpha)}_{n}$ are the coefficients of the Taylor expansion of the
generating function $\omega^{(\alpha)}(z) = \sum_{n=0}^{\infty}\omega^{(\alpha)}_{n}z^n$, $p$ is the convergence
order that depends on the generating function $\omega^{(\alpha)}(z)$.
\end{lemma}

The widely used generating functions $\omega^{(\alpha)}(z)$ in fractional calculus
include the fractional backward difference formula of order $p$ (FBDF-$p$) and
the generalized Newton-Gregory formula of order $p$ (GNGF-$p$), which are given by
\begin{equation}\label{genf:wz}
\omega^{(\alpha)}(z)=
\left\{\begin{aligned}
&\left(\sum_{k=1}^{p} \frac{1}{k}(1-z)^{k}\right)^{\alpha},&\quad \text{FBDF-$p$},\\
&(1-z)^{\alpha} \sum_{k=1}^{p} g_{k-1}(1-z)^{k-1},&\quad \text{GNGF-$p$}.
\end{aligned}\right.\end{equation}
where $g_{0}=1,  g_{1}=\frac{\alpha}{2}, g_{2}=\frac{\alpha^{2}}{8}+\frac{5 \alpha}{24}$, $g_k(k\geq 3)$ can be found in Ref.
\refcite{GuoZeng19}. Interested readers can refer to Ref. \refcite{Lub86} for more generating functions.


Using the  relationship $_{0}^{C}D_{t}^{\alpha}u(t) =\,  _{0}^{RL}D_{t}^{\alpha}(u-u(0))(t)$ (see Ref. \refcite{Pod-B99})
and Lemma \ref{lem2-1}, we can obtain
\begin{equation}\begin{aligned}\label{e1.3}
\left[{}_{0}^{C} D_{t}^{\alpha}u(t)\right]_{t=t_n} =D_{\tau}^{\alpha,m}u^n -   {R}^n,
\end{aligned}\end{equation}
where $R^n$ is the  truncation error in time and
\begin{equation}\begin{aligned}\label{Dtau}
D_{\tau}^{\alpha,m}u^n=\frac{1}{\tau^{\alpha}}\sum_{j=0}^{n}\omega_{n-j}^{(\alpha)} (u^j-u^0)
+\frac{1}{\tau^{\alpha}}\sum_{j=1}^{m}w_{n,j}^{(m)} (u^{j}-u^0).
\end{aligned}\end{equation}
The starting weights $w_{n,j}^{(m)}$ in \eqref{Dtau} are chosen such that
\begin{equation}\label{cond-2}
D_{\tau}^{\alpha,m}u^n=\left[{}_{0}^{C} D_{t}^{\alpha}u(t)\right]_{t=t_n},\quad u=t^{\sigma_k},\quad 1\leq k\leq m.
\end{equation}
If $u$ satisfies \eqref{solution-u} and $\sigma_k=\delta_k,1\leq k\leq m+1$, then   Lemma \ref{lem2-1} and \eqref{cond-2}
yield the truncation error $ {R}^n$ in \eqref{e1.3}, which satisfies
\begin{equation}\label{time-error-2}
{R}^n = O(\tau^{p}t_n^{\sigma_{m+1}-p-\alpha})+O(\tau^{\sigma_{m+1}+1}t_n^{-\alpha-1}),
\end{equation}
where $p$ is the convergence order that depends the   generating function $\omega^{(\alpha)}(z)$.

The quadrature weights $\omega_{n}^{(\alpha)}$ in \eqref{Dtau} can be derived much easily.
For  $\omega^{(\alpha)}(z)$ defined by \eqref{genf:wz}, the recurrence formula (5) in Ref. \refcite{DieFFW06}
can be used to  obtain $\omega_{n}^{(\alpha)}$.
One can also used \eqref{s6-0-1}  to calculate $\omega_{n}^{(\alpha)}$ for
$n\ge n_0$, $n_0$ is a suitable positive integer.

Next, we  give a criterion to select $\sigma_k$  and derive the starting weights $w_{n,j}^{(m)}(1\leq j \leq m)$
when applying  the time discretization method \eqref{Dtau}.

\textbf{1) Determine} $\sigma_k$ \textbf{in \eqref{Dtau}}.

From the construction of the method \eqref{Dtau}, the optimal choice of $\sigma_k$
should be $\sigma_k=\delta_k$, where $\delta_k$ are the regularity indices of
the analytical solution, see \eqref{solution-u}. However, we may not know $\delta_k$
for a generalized nonlinear term $f(u)$.

If  $f(z)$ is sufficiently smooth, then $f(u)$ can be decomposed into   as $f=f_1+f_2$, where
$f_1(u) = f(u_0) + f'(u_0)(u-u_0)$ and $f_2(u)=f(u)-f_1(u)$.
Let $v$ be the solution of the following linear system
\begin{equation}\label{eq:eq-v}
{}_{0}^{C} D_{t}^{\alpha}v=\Delta v+f_1(v), \qquad (x,t)\in \Omega\times (0,T],T>0
\end{equation}
subject to the initial condition $v(x,0)=u_0(x), x \in \bar{\Omega}$ and the homogenous boundary conditions.
Let  $w$ be the solution of the following nonlinear system
  \begin{equation}\label{eq:eq-w}
{}_{0}^{C} D_{t}^{\alpha}w=\Delta w+ f(v+w) - f_1(v), \qquad (x,t)\in \Omega\times (0,T]
\end{equation}
subject to the homogenous initial and boundary conditions.
Then,   the solution of \eqref{e1.1} can be expressed as $u=v+w$.
It is known  that the analytical solution of the linear system \eqref{eq:eq-v}
satisfies $v(t)-v(0)=\sum_{k=1}^m\hat{v}_kt^{\delta_k} + \tilde{v}(t)t^{\delta_{m+1}}$, where  $\delta_k=k\alpha$ for $f(0)=0$
(see Ref. \refcite{Luchko12})
and $\delta_k \in \{\delta_{\ell,j}|\delta_{\ell,j}=\ell+j\alpha,\ell\in \mathbb{Z}^+,j\in \mathbb{N}\}$
for $f(0)\neq 0$ (see Ref. \refcite{CueLubPal06}).
From Ref. \refcite{WangZhou20}, one knows that $w(t)$ has higher regularity than $v(t)$ and
$w(0)={}_{0}^{C} D_{t}^{\alpha}w(t)|_{t=0}=0$.  Therefore,
for a smooth $f(z)$, we have $\delta_1=\alpha$, but $\delta_k$ for $k\geq 2$ need to be determined by further
investigation.

Now, we know that $u(t)=v(t)+w(t)$, the regularity of $v(t)$ is known and $w(t)$ has higher regularity
than $v(t)$. Hence,  it is reasonable to select $\sigma_k$   according to
the regularity of $v(t)$, which is adopted in the current paper, and it performs well;
see numerical results in Section \ref{sec:numer}.

\textbf{2) Derive the starting weights} $w_{n,j}^{(m)}(1\leq j \leq m)$ \textbf{in \eqref{Dtau}}.

For a fixed $n$, the starting weights  $w_{n,j}^{(m)}$  are chosen such that
\eqref{cond-2} holds, which yields the following linear system\cite{Lub86}
 \begin{equation}\label{startw}
\sum_{j=1}^{m} t_j^{\sigma_k}w_{n,j}^{(m)}
=\frac{\Gamma(\sigma_k+1)}{\Gamma(\sigma_k+1-\alpha)}t_n^{\sigma_k-\alpha}
- \sum_{j=0}^{n}\omega_{n-j}^{(\alpha)} t_j^{\sigma_k},\quad 1\leq k \leq m.
\end{equation}
Clearly, \eqref{startw} is a  Vandermonde type system, which  may lead  to
inaccurate starting weights  that may harm the accuracy of the numerical method.\cite{DieFFW06,Lub86,ZengZK17}
Diethelm et al.\cite{DieFFW06} discussed in detail how to solve the linear system \eqref{startw}
and how the starting weights and values affect the accuracy of the numerical method.

Figure \ref{fig1} (a) shows the condition number of \eqref{startw} for different fractional orders $\alpha$ when $\sigma_k=k\alpha$.
We can see that for a smaller $\alpha$, i.e., $\alpha=0.1,0.2$, the condition number of \eqref{startw} increases fast as
$m$ increases up to a certain number, then it  increases slowly. For a larger $\alpha$, i.e., $\alpha=0.8,1$,
the condition number   increases  as $m$ increases.

\begin{figure}
\centering
\begin{minipage}[c]{0.5\textwidth}
\centering
\includegraphics[height=4cm,width=6cm]{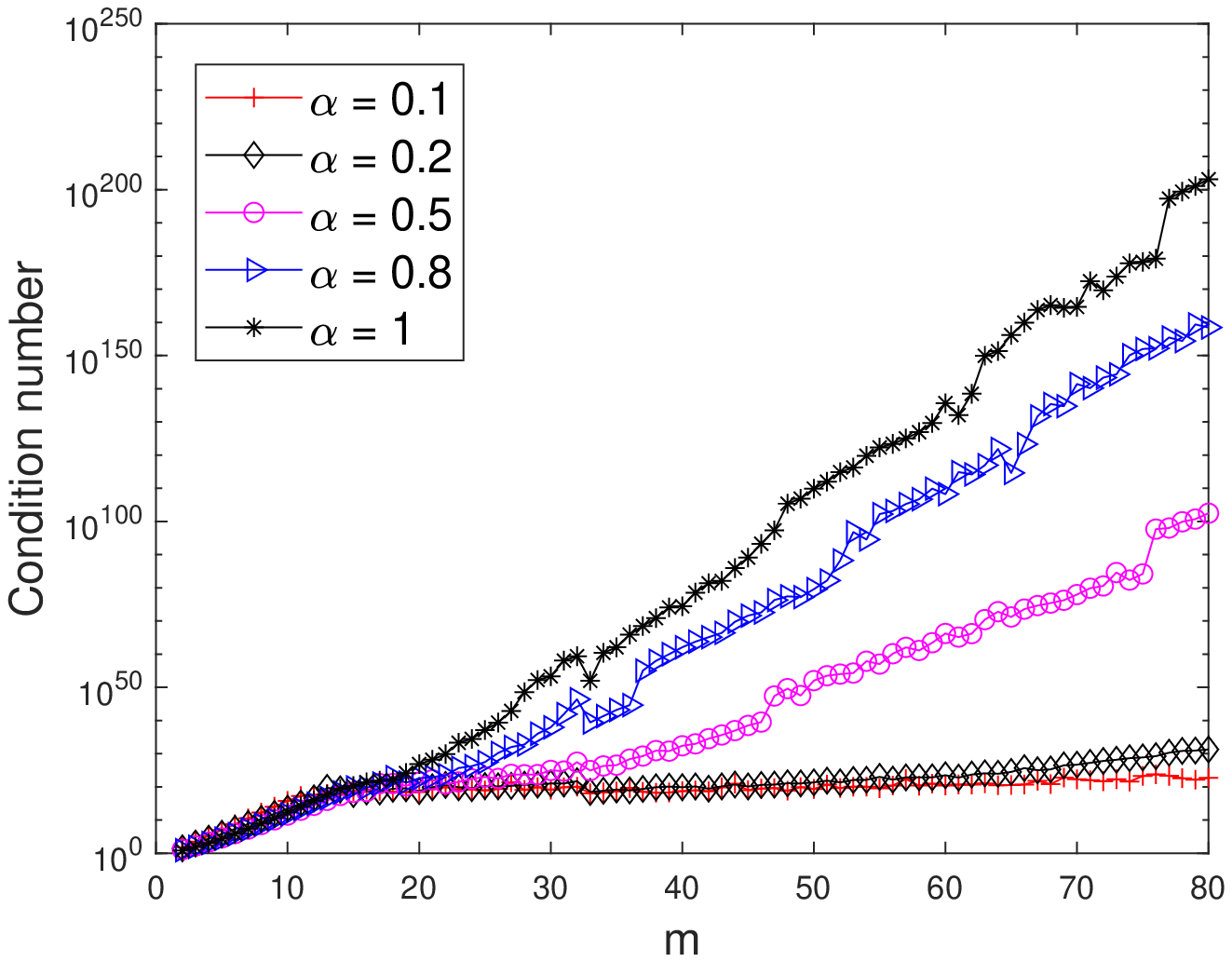}
\par {(a) The condition number of \eqref{startw}.}
\end{minipage}%
\begin{minipage}[c]{0.5\textwidth}
\centering
\includegraphics[height=4cm,width=6cm]{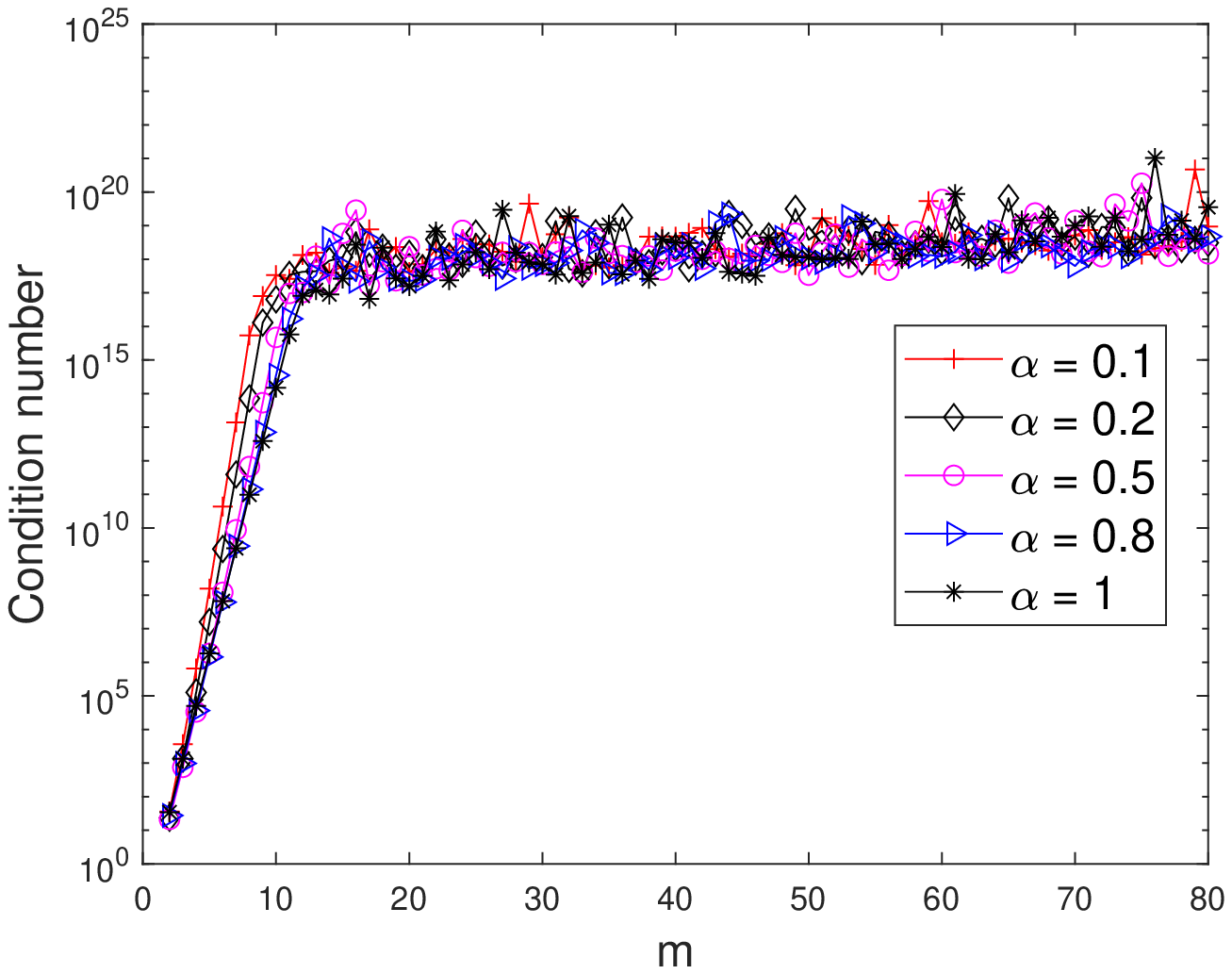}
\par {(b) The condition number of \eqref{startw-2}.}
\end{minipage}
\caption{The condition numbers of different linear systems, $\sigma_k=k\alpha$, $\tau=0.1$.
In practice, a few number of correction terms are enough to obtain accurate numerical solutions.
We take $m\leq 4$ in numerical simulations in Section \ref{sec:numer}, so that
the system \eqref{startw} is relatively well-conditioned and
accurate starting weights can be derived.}
\label{fig1}
\end{figure}


One way to reduce the condition number of \eqref{startw} is to find a suitable preconditioner,
which is not trivial.\cite{DieFFW06}
If we can find a new basis function
$\phi_k(t)$, satisfying
$$\mathrm{span}\{t^{\sigma_1},t^{\sigma_2},\cdots,t^{\sigma_m}\}
=\mathrm{span}\{\phi_1(t),\phi_2(t),\cdots,\phi_m(t)\},$$
then the system \eqref{startw} is equivalent to the following new system
\begin{equation}\label{startw-2}
\sum_{j=1}^{m} \phi_k(t_j) w_{n,j}^{(m)}
= \int_0^{t_n} \frac{(t_n-s)^{-\alpha}}{\Gamma(1-\alpha)}\phi'_k(s)ds
- \sum_{j=0}^{n}\omega_{n-j}^{(\alpha)} \phi_k(t_j),\quad 1\leq k \leq m.
\end{equation}
The condition number of the new system \eqref{startw-2} may become smaller if
the suitable basis functions  $\phi_k(t)$ are chosen.

Figure \ref{fig1} (b) shows the condition number of \eqref{startw-2} for $\tau=0.1$,
where we choose $\phi_k(t)=L^{(0)}_k(t^{\alpha})-L^{(0)}_k(0)$, $L^{(\beta)}_k(t)(\beta > -1)$ is the
generalized Laguerre polynomial.\cite{ShenTW-B11}
We can see that for   $\alpha=0.1,0.2,0.5,0.8,1$, the condition number of \eqref{startw-2}
increases as $m$ increases until it becomes about $10^{20}$ for
$m>10$. We have also tested other fractional orders $\alpha \in (0,1)$ and $\tau<0.1$,
and we have obtained results similar to the ones obtained in  Figure \ref{fig1} (b).
The condition number of \eqref{startw-2} is about $10^{20}$,
the Multiprecision Computing Toolbox for MATLAB\cite{MCTMatlab} can be used  to solve \eqref{startw-2},
which is not costly.

In this paper, we use at most four correction terms in numerical simulations, so that
the system \eqref{startw} is relatively well-conditioned, which can be solved directly.

\subsection{The fully discrete scheme}
Let $\mathcal{T}_h$ be a family of regular (conforming) triangulations of the domain $\bar{\Omega}$
and $h=\max_{K\in \mathcal{T}_h}(\mathrm{diam} K)$.
The linear finite element space $X_h$ is defined as
\begin{equation}\label{Xh-1}
X_h=\{v_h\in H^1_0(\Omega): v_h|_{T}\text{ is a linear function } \forall T\in \mathcal{T}_h\}.
\end{equation}
Define the orthogonal projectors $P_h:L^2(\Omega)\to X_h$ and $\pi_h^{1,0}:H_0^1(\Omega)\to X_h$  as
\begin{align*}
(P_hu, v)&=(u,v),\quad   \quad\forall v \in X_h,\\
(\nabla\pi_h^{1,0}u,\nabla v)&=(\nabla u,\nabla v),\quad  \quad\forall v \in X_h.
\end{align*}
where $(\cdot,\cdot)$ is the inner product in $L^2(\Omega)$ equipped with
the  $L^2$ norm $\|\cdot\|_{L^2(\Omega)}$ and the $L^{\infty}$ norm
$\|\cdot\|_{L^{\infty}(\Omega)}$. Denote by $H^k(\Omega)$ as the Sobolev space equipped with the norm
$\|\cdot\|_{H^k(\Omega)}$, $k\geq 0$. For convenience, we  denote
$\|\cdot\|=\|\cdot\|_{L^{2}(\Omega)}$.

Using \eqref{e1.3}, we can derive the time discretizaion for \eqref{e1.1} as
\begin{equation}\begin{aligned}\label{TD-1}
D_{\tau}^{\alpha,m}u^n =\Delta u^n + f(u^n) + R^n.
\end{aligned}\end{equation}

From \eqref{TD-1}, the fully  discrete Galerkin FEM for
\eqref{e1.1} may be given as: Given $u^0_h=\pi_{h}^{1,0} u_0$, find $u_h^n\in X_h$ for $n\geq 1$, such that
\begin{equation}\label{FLMM-2}
(D_{\tau}^{\alpha,m}u_h^n,v)+(\nabla u_h^n,\nabla v)  =\left(P_{h}f(u^{n}_h),v\right),\quad \forall v\in X_h.
\end{equation}

In order to obtain the starting values $u_h^{n}(1\leq n \leq m)$,
we can let  $n=1,2,\cdots,m$ in \eqref{FLMM-2}, which  yields a  system of equations, its matrix form   reads
\begin{equation}\label{start-uh}
A^{(m)}\left[
  \begin{array}{c}
    (u_h^1,v)\\
    (u_h^2,v)  \\
    \vdots\\
     (u_h^m,v)    \\
  \end{array}
\right]
+ \tau^{\alpha}\left[
  \begin{array}{c}
    (\nabla u_h^1,\nabla v)\\
    (\nabla u_h^2,\nabla v)  \\
        \vdots\\
     (\nabla u_h^m,\nabla v)    \\
  \end{array}
\right]
= \tau^{\alpha}\left[
  \begin{array}{c}
    (P_{h}f(u^{1}_h),v)\\
    (P_{h}f(u^{2}_h),v)  \\
        \vdots\\
     (P_{h}f(u^{m}_h),v)    \\
  \end{array}
\right], \quad \forall v\in X_h,
\end{equation}
where    $A^{(m)}=\Lambda_1 V \Lambda_2V^{-1}$,
$V\in \mathbb{R}^{m\times m}$ with entries $(V)_{i,j}= i^{\sigma_j},1\leq i,j \leq m$,  $\Lambda_1=diag(1,2^{-\alpha},\cdots,m^{-\alpha})$,
$\Lambda_2=diag(\gamma_1,\gamma_2,\cdots,\gamma_m)$ with $\gamma_j=\frac{\Gamma(\sigma_j+1)}{\Gamma(\sigma_j+1-\alpha)},1\leq j \leq m$.

Generally speaking, if  $A^{(m)}+(A^{(m)})^T$ is positive definite,   then
\eqref{start-uh}  permits a unique solution under some suitable conditions,
which is true for $m=1$.
For $m=2$,  one has
\begin{equation*}
A^{(2)}=\frac{1}{2^{\sigma_2}-2^{\sigma_1}}
\left[
  \begin{array}{ll}
    \gamma_12^{\sigma_2}-\gamma_2 2^{\sigma_1}& \gamma_2-\gamma_1\\
    (\gamma_1-\gamma_2)2^{\sigma_1+\sigma_2-\alpha} &   \gamma_22^{\sigma_2-\alpha}- \gamma_12^{\sigma_1-\alpha} \\
  \end{array}
\right].
\end{equation*}
It is a tedious task to find a  condition to guarantee the positive definiteness
of $A^{(2)}+(A^{(2)})^T$, the case for $m\geq 3$ is much more complicated. The well-posedness of \eqref{start-uh}
is not the main goal of this work and is not investigated.

%


In order to obtain a stable and convergent numerical scheme, we need to modify
\eqref{FLMM-2} to obtain a new scheme that works  for all  $m\geq 0$. Obviously,
if   $u_h^{n}(0\leq n \leq m)$ are known, then \eqref{FLMM-2} is well defined
for $n\geq m+1$.

To this end, we can modify \eqref{FLMM-2} as:
Given $u_h^{n}$ for $0\leq n \leq m$,
find $u_h^n\in X_h$ for $n\geq m+1$, such that
\begin{equation}\label{FLMM-3}
(D_{\tau}^{\alpha,m}u_h^n,v)
+(\nabla u_h^n,\nabla v)  =\left(P_{h}f(u^{n}_h),v\right),\quad \forall v\in X_h.
\end{equation}
The existing  numerical methods can be used to obtain $u_h^{n}(1\leq n \leq m)$. For example,
we can solve \eqref{FLMM-2} with one correction term  and a smaller time step size to
obtain   $u_h^{n}(1\leq n \leq m)$, which is adopted  in   numerical simulations
when analytical solution is unavailable.
The convergence of \eqref{FLMM-3} is given in Theorem \ref{thm2-1},
in which we display how the starting values affect the numerical solutions of \eqref{FLMM-3}.

In order to prove the convergence of  \eqref{FLMM-3}, we define  the generating functions
$a^{(\alpha)}(z)$ and $b(z)$  as
\begin{eqnarray}
&&a^{(\alpha)}(z)=(1-z)^{\alpha}=\sum_{n=0}^{\infty}a^{(\alpha)}_{n}z^n,\label{eq:zeng-a}\\
&&b(z)=a^{(\alpha)}(z)/\omega^{(\alpha)}(z)=\sum_{n=0}^{\infty}b_{n}z^n.\label{eq:omega}
\end{eqnarray}
Introduce the following notations:
\begin{eqnarray}
&&\hat{b}(z) =  \sum_{n=0}^{\infty}\hat{b}_nz^n,
\qquad\qquad  \hat{b}_n = |b_n|,n\geq 0; \label{eq:zeng-bh}\\
&&c(z)=\left(2b_0-\hat{b}(z)\right){a}^{(-\alpha)}(z)=\sum_{n=0}^{\infty}c_nz^n,\,\,\,
c_n= 2b_0{a}^{(-\alpha)}_n -\sum_{j=0}^{n}\hat{b}_j{a}^{(-\alpha)}_{n-j}. \label{eq:zeng-d}
\end{eqnarray}
The following assumptions are used in the convergence analysis:
\begin{eqnarray}
 && b_0> 0, |b_n| \lesssim n^{-\alpha-1},
 \sum_{n=1}^{\infty}|b_n| \leq b_0;  \label{assumption-a}\\
 && c_0>0,c_n\geq0,n>0, \label{assumption-c}
\end{eqnarray}
where $A\lesssim B$ means  there exists a positive constant $C$
independent of $\tau,h$, and any positive integer $n$, such that
$A \leq CB.$  In the rest of this paper, $C_k,k\in\mathbb{N}$ are  generic positive constants  independent of
$\tau,h$ and any positive integer $n$.

The  assumptions \eqref{assumption-a}--\eqref{assumption-c} are verified in Section \ref{sec-5}
when the specific time discretization method   is used.
We have the following theorems, the proofs of which are given in Section \ref{sec-4}.
\begin{theorem}\label{thm2-2}
Suppose that $u_0\in H^{2}(\Omega) \cap H^1_0(\Omega)$,
 $u$  is the solution  of \eqref{e1.1} satisfying \eqref{solution-u},
$u(t)\in H^{2}(\Omega) \cap H^1_0(\Omega)$, $f(u(t))\in H^{2}(\Omega)$,
and $|f'(z)|\lesssim 1$ for  $|z|\lesssim  1$.
Let $u_h^n(1\leq n \leq n_T)$ be the solution of  \eqref{FLMM-2},
$\sigma_k=\delta_k,k=1,2$, $m=0,1$.
If    the assumptions \eqref{assumption-a} and \eqref{assumption-c} hold,   then
\begin{equation}\label{eq:zeng-7}\begin{aligned}
\|u^n_h-u(\cdot,t_n)\|&\lesssim t_{n}^{\alpha/2}h^{2}+\tau^{\sigma_{m+1}-\alpha/2}
\big(\ell_n^{(\sigma_{m+1})}\big)^{1/2},
\end{aligned}\end{equation}
where $ \ell_n^{(\sigma)}$ is defined by
\begin{equation}\label{eq:ell-n}
 \ell_n^{(\sigma)}
=\left\{\begin{aligned}
&n^{\max\{\alpha-1,2\sigma-2p-\alpha\}},&& \sigma \ne  p+\alpha -1/2,\\
&n^{\alpha-1}\ln(n),\quad &&  \sigma= p+\alpha -1/2.
\end{aligned}\right.
\end{equation}
Furthermore, if $\sigma_{m+1} < p+\alpha-1/2$, then
\begin{equation}\label{eq:zeng-7-22}\begin{aligned}
\|u^n_h-u(\cdot,t_n)\|&\lesssim t_{n}^{\alpha/2}h^{2}+\tau^{\sigma_{m+1}-\alpha/2}n^{(\alpha-1)/2}.
\end{aligned}\end{equation}
\end{theorem}

\begin{theorem}\label{thm2-1}
Suppose that $u_0\in H^{2}(\Omega) \cap H^1_0(\Omega)$,
$u$  is the solution  of \eqref{e1.1} satisfying \eqref{solution-u},
$u(t)\in H^{2}(\Omega) \cap H^1_0(\Omega)$, $f(u(t))\in H^{2}(\Omega)$,
and $|f'(z)|\lesssim 1$ for  $|z|\lesssim  1$.
Let $u_h^n(1\leq n \leq n_T)$ be the solution of  \eqref{FLMM-3}, $\sigma_k=\delta_k,1\leq k \leq m+1$.
If the assumptions  \eqref{assumption-a} and \eqref{assumption-c} hold, then
\begin{equation}\label{eq:zeng-7-2}\begin{aligned}
\|u^n_h-u(\cdot,t_n)\|& \le   C_1h^{2}
+C_2\tau^{\sigma_{m+1}-\alpha/2} \big({\ell_n^{(\sigma_{m+1})}}\big)^{1/2}+C_3\mathcal{E}^n,
\end{aligned}\end{equation}
where
$ \ell_n^{(\alpha,\sigma)}$ is defined by \eqref{eq:ell-n},  and $\mathcal{E}^n$ is   given by
\begin{equation}\label{En}
\mathcal{E}^n=\tau^{\alpha/2}\Big(\sum_{k=1}^m \|e^k/\tau^{\alpha}\|\Big)
\big(\ell_n^{(\sigma_{m})}\big)^{1/2},\quad e^k=(u^k-u^0)-(u_h^k-u_h^0).
\end{equation}

\end{theorem}

\begin{remark}
We keep $\mathcal{E}^n$ in \eqref{eq:zeng-7-2} in order
to show how the errors of the starting values $u_h^k(1\leq k \le m)$ influence the accuracy of numerical solutions far from the origin.
If the starting values  are accurate enough, i.e., $\mathcal{E}^n$ is sufficiently small,
then we can drop $\mathcal{E}^n$, so that  \eqref{eq:zeng-7-2} can be simplified as
\begin{equation}\label{eq:zeng-7-2-2}
\|u^n_h-u(\cdot,t_n)\|\lesssim    h^{2} +\left\{\begin{aligned}
&\tau^{\sigma_{m+1}-\alpha+1/2}t_n^{(\alpha-1)/2},&\quad \sigma_{m+1}< p+\alpha - 1/2, \\
&\tau^p\ln(n)t_n^{\sigma_{m+1}-\alpha/2-p},&\quad \sigma_{m+1}= p+\alpha - 1/2,\\
&\tau^pt_n^{\sigma_{m+1}-\alpha/2-p},&\quad \sigma_{m+1}> p+\alpha - 1/2.
\end{aligned}\right.\end{equation}
\end{remark}

\begin{remark}\label{remark-2}
The error bound  \eqref{eq:zeng-7-2} also shows that
$\mathcal{E}^n$  may harm the accuracy of numerical solutions
if $\ell_n^{(\sigma_{m})}$ is too large.
It is easy to verify that if $\sigma_m\geq  p+\alpha-1/2$, then
$\ell_n^{(\sigma_{m})}$ increases as $\sigma_m$ increases,
which makes the error induced by the starting values  harm the accuracy of numerical solutions,
especially when $\sigma_m$ is sufficiently large;
see Refs. \refcite{DieFFW06} and \refcite{Lub86}.
\end{remark}

\begin{remark}\label{remark-3}
If  $\Omega$ is a rectangular domain, i.e.,  $\Omega=I^x \times I^y, I^x=(x_L,x_R), I^y=(y_L,y_R)$,
then the high-order bilateral element
of order $r$ can be used,  and the corresponding finite element space $X_h$ can be defined by
\begin{equation}\label{Xh-r}
X_h=X_{h_x}^x \otimes X_{h_y}^y,
\end{equation}
where
$$X_{h_\theta}^\theta=\{v:v|_{I^\theta_i}\in \mathbb{P}_r(I^\theta_i)\cap H_0^1(I^\theta)\},\quad \theta=x,y.$$
Here $\mathbb{P}_r(I^\theta_i)$ denotes the polynomial space of order $r$ on $I^\theta_i=[\theta_{i-1}-\theta_{i}]$,
$\theta_i=\theta_L + (i-1)h_\theta,h_\theta=(\theta_{R}-\theta_L)/N_\theta$,
$N_\theta\in \mathbb{N}$. For the two dimensional problem on the rectangular
domain  $\Omega=I^x \times I^y$, if the finite element space  \eqref{Xh-1} is replaced by \eqref{Xh-r},
then Theorem \ref{thm2-2} and \ref{thm2-1} hold,  but the  convergence rate in space changes to $O(h^{r+1})$.
\end{remark}
\section{Error estimate}\label{sec-4}
In this section, we show how to apply the generalized discrete Gr\"{o}nwall inequality
 to prove Theorems \ref{thm2-2} and \ref{thm2-1}.
\subsection{Lemmas}
Some useful lemmas  are introduced in this subsection.
\begin{lemma}[see Ref. \refcite{ZengLLT15}]\label{lem3-1}
Let  $a^{(\beta)}(z)=(1-z)^{\beta}=\sum_{n=0}^{\infty}a^{(\beta)}_nz^n$, $\beta\in \mathbb{R}$, and $0\leq\alpha\leq 1$.
Then
\begin{eqnarray}
&&a^{(\alpha)}_0=1, \,a^{(\alpha)}_n=O(n^{-\alpha-1}),a^{(\alpha)}_n\leq 0\,\, \text{ for }\,\, n>0,\quad
 0<-\sum_{n=1}^{\infty}a^{(\alpha)}_n\leq 1;\label{eq:zeng-a2}\\
&&a^{(-\alpha)}_0=1, \,\,a^{(-\alpha)}_n=O(n^{\alpha-1}),\,\, a^{(-\alpha)}_n\geq 0
\text{ for } n>0;\label{eq:zeng-a2-2}\\
&&\sum_{k=0}^na^{(\alpha)}_{k}a^{(-\alpha)}_{n-k}=0\,\,\text{ for } \,\, n>0.\label{eq:zeng-a2-3}
\end{eqnarray}
\end{lemma}
The  equation \eqref{eq:zeng-a2-3} can be obtained from $a^{(\alpha)}(z)a^{(-\alpha)}(z)=1$.

\begin{lemma}\label{lem4-2}
Let $\sigma \in \mathbb{R}$ and $0< \alpha \leq 1$. Then
\begin{eqnarray}
&&\sum_{j=1}^{n-1}(n-j)^{-\alpha-1}j^{\sigma} \lesssim n^{\max\{-\alpha-1,\sigma\}},\label{eq:zeng-a2-3-2}\\
&&\sum_{j=1}^{n}a_{n-j}^{(-\alpha)}j^{\sigma}\lesssim
 \left\{\begin{aligned}
& n^{\max\{\alpha-1,\sigma+\alpha\}},&\qquad \sigma \ne -1,\\
& n^{\alpha-1}\ln(n),&\qquad \sigma = -1.
\end{aligned}\right.\label{eq:zeng-a2-3-3}
\end{eqnarray}
\end{lemma}
The proof of Lemma \ref{lem4-2} is given in \ref{sec-D}.

%

\begin{lemma}[see Ref. \refcite{AL-MaskariKaraa19} Discrete fractional Gr\"{o}nwall inequality]
	\label{le3.1}
Assume that $\alpha>0$, $A,B\geq0$, and  $\delta<1$.
Let $z_n$, $0\leq n\leq K$, be a sequence of non-negative real numbers satisfying
\begin{equation*}
z_n\leq C\tau^\alpha\sum_{j=0}^{n-1}(n-j)^{\alpha-1}z_j+(A+B\log(n))t_n^{-\delta},\quad  1\leq n\leq K,
\end{equation*}
where  $C>0$ is bounded independent of $\tau$ and $n$. Then
 $z_n\lesssim (A+B\log(n))t_n^{-\delta}.$
\end{lemma}

The case of $B=\delta=0$  in Lemma \ref{le3.1} is the original version of the
discrete fractional Gr\"{o}nwall inequality in Ref. \refcite{DixonMcKee86}.
For $B=0$, Lemma \ref{le3.1} is equivalent to
the discrete fractional Gr\"{o}nwall inequality (see Ref. \refcite{GonPal99} Lemma 2.1).

From Lemma \ref{le3.1}, we can deduce the following corollary, which will be used in the
convergence analysis instead of Lemma \ref{le3.1} for convenience.
\begin{corollary}\label{corollary-1}
Assume that $A,B\geq 0,C>0$, $\delta<1$, and $\alpha> 0$.
Let $z_n$, $0\leq n\leq K$, be a sequence of non-negative real numbers satisfying
\begin{equation*}
z_n\leq C\tau^\alpha\sum_{j=0}^{n}a^{(-\alpha)}_{n-j}z_j+(A+B\log(n))t_n^{-\delta},\quad  1\leq n\leq K.
\end{equation*}
If $C\tau^{\alpha}\leq 1/2$, i.e., $\tau \leq (2C)^{-1/\alpha}$,
then $z_n\lesssim (A+B\log(n))t_n^{-\delta}.$
\end{corollary}
\begin{proof}
Using  $a_0^{(-\alpha)}=1$, $a_n^{(-\alpha)}\lesssim n^{\alpha-1}$ for $n\geq 1$, the condition
$C\tau^{\alpha}\leq 1/2$, and Lemma \ref{le3.1} yields the
desired result, which ends the proof.
\end{proof}



\begin{lemma}[see Ref. \refcite{BrennerSR08-B}]\label{le3.3}
Let  $0\leq s\leq 1, s\leq   r$. Then the following estimates hold
\begin{eqnarray*}
\|u-\pi_{h}^{1,0}u\|_{H^s(\Omega)} &\lesssim& h^{r-s}\|u\|_{H^r(\Omega)},\qquad u\in H^{r}(\Omega)\cap H_0^1(\Omega),\\
\|u-P_{h}u\|_{L^2(\Omega)}  &\lesssim& h^{r}\|u\|_{H^r(\Omega)},\qquad  u\in H^{r}(\Omega).
\end{eqnarray*}
\end{lemma}



\subsection{Proofs of Theorems \ref{thm2-2} and \ref{thm2-1}}
For the sequence $\{u^n\}_{n=1}^{\infty},u^n \in L^2(\Omega)$, we
define the following notations:
\begin{eqnarray}
&&\mathcal{A}^{\alpha,m}_{\tau}u^n=\frac{1}{\tau^{\alpha}}\sum_{j=m+1}^na_{n-j}^{(\alpha)}u^j, \label{A-alf}\\
&&\mathcal{B}^{\alpha,m}u^n=\sum_{j=m+1}^nb_{n-j}u^j,\qquad\qquad
 \widehat{\mathcal{B}}^{\alpha,m}u^n=\sum_{j=m+1}^n\hat{b}_{n-j}u^j,   \label{B-alf}\\
&&W^{(m)}_{n,k}=\sum_{j=m+1}^nb_{n-j}\left(w_{j,k}^{(m)}+  \omega_{j-k}^{(\alpha)}\right),  \label{eq:Wnm} \\
&&\mathcal{D}^{\alpha,m}_{\tau}u^n
=\frac{1}{\tau^{\alpha}}\sum_{j=m+1}^na_{n-j}^{(\alpha)}(u^j-u^0)
+ \frac{1}{\tau^{\alpha}}\sum_{j=1}^mW_{n,j}^{(m)} (u^j-u^0), \label{A-alf-m}
\end{eqnarray}



\begin{lemma}
The following statements hold:
\begin{eqnarray}
\mathcal{A}_{\tau}^{\alpha_1,m}\mathcal{A}_{\tau}^{\alpha_2,m}&=&\mathcal{A}_{\tau}^{\alpha_1+\alpha_2,m},
\qquad \alpha_1,\alpha_2 \in \mathbb{R},\label{alphabeta}\\
\mathcal{B}^{\alpha,m}D_{\tau}^{\alpha,m} &=&\mathcal{D}^{\alpha,m}_{\tau},\label{BalfD1}\\
\mathcal{A}^{-\alpha,m}_{\tau}\widehat{\mathcal{B}}^{\alpha,m}\|u^{n}\|^2
&\lesssim& \mathcal{A}^{-\alpha,m}_{\tau} \|u^{n}\|^2, \label{BleqA}\\
(\mathcal{A}^{\alpha,m}_{\tau}u^n,2u^n)
&\geq& \mathcal{A}^{\alpha,m}_{\tau}\|u^n\|^2, \label{sec3:1}\\
 (\mathcal{B}^{\alpha,m}u^n,2u^n)
&\geq&   2b_0\|u^n\|^2-\widehat{\mathcal{B}}^{\alpha,m}\|u^n\|^2.\label{sec3:2}
\end{eqnarray}

\end{lemma}
\begin{proof}
Eq. \eqref{alphabeta}  can be derived from $(1-z)^{\alpha_1}(1-z)^{\alpha_2}=(1-z)^{\alpha_1+\alpha_2}$,
Eq. \eqref{BalfD1} is derived from $b(z)\omega^{(\alpha)}(z)=(1-z)^{\alpha}$  and \eqref{eq:Wnm}.
The equality \eqref{BleqA} can be deduced from \eqref{assumption-a}, \eqref{eq:zeng-a2-2}, and \eqref{eq:zeng-a2-3-2}.
The Cauchy--Schwarz inequality $2(u,v)\leq \|u\|^2 + \|v\|^2$ and \eqref{eq:zeng-a2} (or \eqref{assumption-a})
yield  \eqref{sec3:1} (or \eqref{sec3:2}); see Ref. \refcite{WangZou19}. The proof is completed.
\end{proof}

\begin{lemma}\label{lem-10}
Let $0<\alpha\leq 1$, $a_n^{(-\alpha)}$ and $W_{n,k}^{(m)}$ be defined by \eqref{eq:zeng-a2-2} and \eqref{eq:Wnm}, respectively. Then
\begin{equation}\label{eq:Wnk-2}
\sum_{j=m+1}^na_{n-j}^{(-\alpha)}(W_{j,k}^{(m)})^2 \lesssim  \ell_n^{(\sigma_{m})},
 \quad 1\leq k \leq m,
\end{equation}
where   $\ell_n^{(\sigma_{m})}$ is  by
\eqref{eq:ell-n}.
\end{lemma}
The proof of Lemma \ref{lem-10} is given in \ref{sec-D}.

For simplicity, we assume  that the nonlinear function $f(z)$ satisfies the global Lipschitz condition, i.e.,
\begin{equation}\label{Lips}
|f(z_1)-f(z_2)| \lesssim |z_1-z_2|.
\end{equation}
If $f(z)$ satisfies the local Lipschitz condition, then the temporal-spatial splitting technique
can be used to  analyze the convergence; see Refs. \refcite{LiSun2013} and \refcite{LiZhangZhang18}.
An alternative way to deal with the nonlinear term
is to construct a function $\bar{f}(z)$, satisfying the
global Lipschitz condition and
$\bar{f}(z)=f(z)$ for $|z|\leq 1+ \max_{0\le t \le T}\|u(t)\|_{L^\infty(\Omega)}$, where $u(t)$
is the solution of  \eqref{e1.1}. Replacing $f(u)$  with $\bar{f}(u)$ in \eqref{FLMM-3}
(or \eqref{FLMM-2}), one can obtain a new scheme, whose solution
is also the solution of   \eqref{FLMM-3}  (or \eqref{FLMM-2}); see Refs.   \refcite{BaoCai2012} and \refcite{WangZhou20}.

\subsubsection{Proof of Theorem \ref{thm2-2}}
By \eqref{B-alf}, \eqref{A-alf-m}, and \eqref{BalfD1}, the scheme \eqref{FLMM-3}  can be reformulated as
\begin{equation}\label{eq:zeng-5}
(\mathcal{D}^{\alpha,m}_{\tau}u_h^n,v)+(\mathcal{B}^{\alpha,m} \nabla u_h^n,\nabla v)
=\left(\mathcal{B}^{\alpha,m}F_h^n,v\right),\quad v\in X_h,
\end{equation}
where $F_h^n = P_{h}f(u^n_h)$.  Similarly, Eq. \eqref{TD-1}  can be written as
\begin{equation}\label{eq:zeng-1}\begin{aligned}
\mathcal{D}^{\alpha,m}_{\tau}u^n
=\mathcal{B}^{\alpha,m} \Delta u^n +\mathcal{B}^{\alpha,m} F^n  +\mathcal{B}^{\alpha,m}R^n,\qquad F^n = f(u^n).
\end{aligned}\end{equation}

Let $\xi_{h}^{n}=\pi_{h}^{1,0}u^{n}- u_{h}^{n} $ and $ \eta_h^n = \pi_{h}^{1,0}u^{n}-u^n$.
From \eqref{eq:zeng-5} and \eqref{eq:zeng-1}, we can obtain the following error equation
\begin{equation}\label{e3.22}
 (\mathcal{D}_{\tau}^{\alpha,m}\xi^n_h,v)
+ (\mathcal{B}^{\alpha,m}\nabla \xi^n_h,\nabla v)
= (\mathcal{B}^{\alpha,m}\widetilde{F}_h^n,v)+(G^n+\mathcal{B}^{\alpha,m} {R}^n,v), \quad v\in X_h,
\end{equation}
where $\widetilde{F}_h^n=P_{h} f(\pi_{h}^{1,0}u^{n})-P_{h}f(u_h^n)$ and
\begin{equation}\label{eq:Gn}
G^n = \mathcal{D}_{\tau}^{\alpha,m}\eta_h^n
- \sum_{j=m+1}^{n}b_{n-j} (P_{h}f(\pi_{h}^{1,0} u^j)-f(u^j)).
\end{equation}
From Lemma \ref{le3.3} and \eqref{Lips},  one has
\begin{eqnarray}\label{e3.38-2}
\|\widetilde{F}_h^n\|=\|P_{h} ( f(\pi_{h}^{1,0}u^{n})-f(u_h^n))\|\lesssim \|f(\pi_{h}^{1,0}u^{n})-f(u_h^n)\|\lesssim \|\xi^n\|.
\end{eqnarray}
We can similarly derive
$$\|P_{h}f(\pi_{h}^{1,0} u^j)-f(u^j)\|\leq \|P_{h}f(\pi_{h}^{1,0} u^j)-P_{h}f(u^j)\| + \|P_{h}f(u^j)-f(u^j)\|
\lesssim h^{2},$$
which, together with  \eqref{assumption-a} and $\|\mathcal{D}_{\tau}^{\alpha,m}\eta_h^n\|\lesssim h^{2}$, yields
\begin{equation}\label{eq:Gn-2}
\|G^n\|\lesssim \|\mathcal{D}_{\tau}^{\alpha,m}\eta_h^n\|
+ \sum_{j=m+1}^{n}\hat{b}_{n-j} \|P_{h}f(\pi_{h}^{1,0} u^j)-f(u^j)\| \lesssim  h^{2}.
\end{equation}
From  \eqref{time-error-2}, \eqref{assumption-a}, and \eqref{eq:zeng-a2-3-2}, we have
\begin{equation}\label{Rn-2}\begin{aligned}
\|\mathcal{B}^{\alpha,m} {R}^n\|
\lesssim &\tau^{\sigma_{m+1}-\alpha}\sum_{j=m+1}^{n}(n+1-j)^{-\alpha-1}
\Big(j^{\sigma_{m+1}-\alpha-p}+j^{-\alpha-1}\Big)\\
\lesssim &\tau^{\sigma_{m+1}-\alpha}\left(n^{\sigma_{m+1}-\alpha-p}+n^{-\alpha-1}\right).
\end{aligned}\end{equation}
Combining \eqref{eq:Gn-2}, \eqref{Rn-2}, and \eqref{eq:zeng-a2-3-3} yields
\begin{eqnarray}
\mathcal{A}^{-\alpha,m}_{\tau}(\|G^n\|^2+\|\mathcal{B}^{\alpha,m} {R}^n\|^2)&\lesssim& t_{n}^{\alpha}h^{4}
+\tau^{2\sigma_{m+1}-\alpha} \ell_n^{(\sigma_{m+1})}\label{NormRn},
\end{eqnarray}
where $\ell_n^{(\sigma_{m+1})}$ is defined  by \eqref{eq:ell-n}.

\begin{proof}
Let   $\Theta^n = \sum_{j=1}^m W_{n,j}^{(m)} \frac{\xi_h^j}{\tau^{\alpha}}.$
By \eqref{A-alf}, \eqref{A-alf-m}, and $\xi_h^0=0$, we rewrite \eqref{e3.22} as
\begin{equation}\label{error-eq}
 (\mathcal{A}_{\tau}^{\alpha,m}\xi^n_h,v)
+ (\mathcal{B}^{\alpha,m}\nabla \xi^n_h,\nabla v)
= (\mathcal{B}^{\alpha,m}\widetilde{F}_h^n,v)+(G^n+\mathcal{B}^{\alpha,m} {R}^n - \Theta^n,v),v\in X_h.
\end{equation}

The proof is finished in two steps.

\textbf{Step 1)}
Letting $n=1$ and $v=2\xi_h^1$ in  \eqref{error-eq} yields
\begin{equation}\label{e3.25-3}
2\gamma_1\|\xi_h^1\|^2 + 2{\tau^{\alpha}}\|\nabla \xi_h^1\|^2 =  {\tau^{\alpha}}(\widetilde{F}_h^1+G^1+R^1, 2\xi_h^1),
\end{equation}
where $\gamma_1=w_0^{(\alpha)}>0$ for $m=0$ and
$\gamma_1=\frac{\Gamma(\sigma_1+1)}{\Gamma(\sigma_1+1-\alpha)}>0$ for $m=1$.
Applying the Cauchy-Schwarz inequality and $\|\widetilde{F}_h^1\|\lesssim \|\xi^1_h\|$ (see \eqref{e3.38-2}), we obtain
\begin{equation}\label{e3.25-4}\begin{aligned}
2\gamma_1\|\xi_h^1\|^2 + 2\tau^{\alpha}\|\nabla \xi_h^1\|^2
&\leq C_1\tau^{2\alpha}(\|\xi_h^1\|^2 +\|G^1\|^2+\|R^1\|^2) + \gamma_1\|\xi_h^1\|^2.
\end{aligned}\end{equation}
If  $(\gamma_1-2C_1\tau^{2\alpha}) \geq 0$, i.e., $\tau\leq \left(2^{-1}\gamma_1/C_1\right)^{1/({2\alpha})}$,
then  \eqref{e3.25-4} leads to
\begin{equation}\label{e3.25-4-0}\begin{aligned}
\gamma_1\|\xi_h^1\|^2 + 4\tau^{\alpha}\|\nabla \xi_h^1\|^2
&\leq -(\gamma_1-2C_1\tau^{2\alpha})\|\xi_h^1\|^2 + 2C_1\tau^{2\alpha} (\|G^1\|^2+\|R^1\|^2)\\
&\leq  2C_1\tau^{2\alpha} (\|G^1\|^2+\|R^1\|^2).
\end{aligned}\end{equation}
Combining  $R^1=O(\tau^{\sigma_2-\alpha})$, \eqref{eq:Gn-2},  and \eqref{e3.25-4-0} yields
\begin{equation}\label{e3.25-4-2}
\|\xi_h^1\|^2 \lesssim \tau^{2\alpha}\left(\|G^1\|^2+ \|R^1\|^2\right)
 \lesssim \tau^{2\alpha}  h^{4} + \tau^{2\sigma_{m+1}}.
\end{equation}

\textbf{Step 2)}  For $ n\geq m+1$, we can
take $v=2\xi_h^n$ in \eqref{error-eq} and use \eqref{sec3:1}--\eqref{sec3:2} to obtain
\begin{equation}\label{e3.29}\begin{aligned}
&\mathcal{A}^{\alpha,m}_{\tau}\|\xi_h^n\|^2+2b_0\|\nabla \xi_h^n\|^2
-\widehat{\mathcal{B}}^{\alpha,m}\|\nabla \xi_h^n\|^2 \\
\leq&  (\mathcal{B}^{\alpha,m}\widetilde{F}_h^n,2\xi_h^n)
+ (G^n+\mathcal{B}^{\alpha,m} R^n- \Theta^n,2\xi_h^n)\\
\leq& \sum_{j=m+1}^{n}\hat{b}_{n-j}(\|\widetilde{F}_h^j\|^2+ \|\xi_h^n\|^2)
+ \|G^n+\mathcal{B}^{\alpha,m}R^n-\Theta^n\|^2+  \|\xi_h^n\|^2  \\
\leq&  C_2\widehat{\mathcal{B}}^{\alpha,m}\|\xi_h^n\|^2 + (1+2b_0)\|\xi_h^n\|^2
+3\rho^n,
\end{aligned}\end{equation}
where  we used $\sum_{j=1}^{n}\hat{b}_{n-j}\leq 2b_0$ and $\|\widetilde{F}_h^j\|\lesssim \|\xi_h^j\|$, and $\rho^n$ is given by
\begin{equation}\label{rhoRe}
\rho^n=\|G^n\|^2 +\|\mathcal{B}^{\alpha,m}R^n\|^2 +\|\Theta^n\|^2.
\end{equation}

Applying $\mathcal{A}^{-\alpha,m}_{\tau}$  on both sides of \eqref{e3.29},
using $\mathcal{A}^{-\alpha,m}_{\tau}\mathcal{A}^{\alpha,m}_{\tau}\|\xi_h^n\|^2=
\mathcal{A}^{0,m}_{\tau}\|\xi_h^n\|^2=\|\xi_h^n\|^2$
(see \eqref{alphabeta}) and  \eqref{eq:zeng-d}, we obtain
\begin{equation}\label{e3.29-2}\begin{aligned}
 &\|\xi_h^n\|^2 + \tau^{\alpha}\sum_{j=m+1}^{n}c_{n-j}\|\nabla \xi_h^j\|^2\\
=&\mathcal{A}^{-\alpha,m}_{\tau}\mathcal{A}^{\alpha,m}_{\tau}\|\xi_h^n\|^2
+\mathcal{A}^{-\alpha,m}_{\tau} (2b_0 \|\nabla \xi_h^n\|^2
-  \widehat{\mathcal{B}}^{\alpha,1}\|\nabla \xi_h^n\|^2)\\
\leq& C_2\mathcal{A}^{-\alpha,m}_{\tau}\widehat{\mathcal{B}}^{\alpha,m}\|\xi_h^n\|^2
+(1+2b_0)\mathcal{A}^{-\alpha,m}_{\tau} \|\xi_h^n\|^2+3 \mathcal{A}^{-\alpha,m}_{\tau} \rho^n.
\end{aligned}\end{equation}
By $c_n\geq 0$ (see \eqref{assumption-c}) and
$\mathcal{A}^{-\alpha,m}_{\tau}\widehat{\mathcal{B}}^{\alpha,m}\|\xi_h^n\|^2
\lesssim \mathcal{A}^{-\alpha,m}_{\tau}\|\xi_h^n\|^2$ (see \eqref{BleqA}), we obtain
\begin{equation}\label{e3.29-4}
 \|\xi_h^n\|^2 + c_0\tau^{\alpha}\|\nabla \xi_h^n\|^2
\leq C_3 \mathcal{A}^{-\alpha,m}_{\tau}
\left(\|\xi_h^n\|^2+ c_0\tau^{\alpha}\|\nabla \xi_h^n\|^2\right)+3 \mathcal{A}^{-\alpha,m}_{\tau}\rho^n.
\end{equation}
If $\tau \leq (2C_3)^{-1/\alpha}$, then we can
apply Corollary \ref{corollary-1}   to obtain
\begin{equation}\label{e4.3111}\begin{aligned}
\|\xi^n_h\|^2  &\lesssim  \|\xi_h^n\|^2 + c_0\tau^{\alpha}\|\nabla \xi_h^n\|^2 \lesssim   \mathcal{A}^{-\alpha,m}_{\tau}\rho^n.
\end{aligned}\end{equation}
From \eqref{eq:Wnk-2} and \eqref{NormRn},      we have
\begin{equation}\label{Arhon}\begin{aligned}
 \mathcal{A}^{-\alpha,m}_{\tau}\rho^n &= \mathcal{A}^{-\alpha,m}_{\tau}
(\|G^n\|^2+\|\mathcal{B}^{\alpha,m}R^n\|^2+\|\Theta^n\|^2)\\
&\lesssim t_{n}^{\alpha}h^{4}+\tau^{2\sigma_{m+1}-\alpha}  \ell_n^{(\sigma_{m+1})}
+ \tau^{\alpha}\ell_n^{(\sigma_{m})} \sum_{k=1}^m\|\xi^k/\tau^{\alpha}\|^2.
\end{aligned}\end{equation}
Combing \eqref{e4.3111}, \eqref{Arhon}, and  \eqref{e3.25-4-2},
and using $ \ell_n^{(\sigma_{m})}\leq \ell_n^{(\sigma_{m+1})}$,
we have
\begin{equation}\label{e4.31111}\begin{aligned}
\|\xi^n_h\|^2&\lesssim  t_{n}^{\alpha}h^{4}+\tau^{2\sigma_{m+1}-\alpha}  \ell_n^{(\sigma_{m+1})}.
\end{aligned}\end{equation}
Using \eqref{e4.3111},
$\|u_h^n-u(\cdot,t_n)\| \leq\|\eta_h^n\| + \|\xi_h^n\|$,  and $\|\eta_h^n\|\lesssim h^{2}$
yields \eqref{eq:zeng-7}, which completes the  proof.
\end{proof}
\subsubsection{Proof of Theorem \ref{thm2-1}}
Similar to \eqref{e3.22},
the error equation of \eqref{FLMM-3} reads as
\begin{equation}\label{eq43-17}
(\mathcal{A}^{\alpha,m}_{\tau}\xi_h^n,v)+( \mathcal{B}^{\alpha,m} \nabla \xi_h^j,\nabla v)
= (\mathcal{B}^{\alpha,m}\widetilde{F}_h^n,v)\\
+(H^n + \mathcal{B}^{\alpha,m}  {R}^n- \Phi^n ,v),
\end{equation}
where   $\Phi^n=\sum_{j=1}^m W_{n,j}^{(m)}  ({e^j}/{\tau^{\alpha}})$, $e^j=(u^j-u^0)-(u_h^j-u_h^0)$, and
\begin{equation}\begin{aligned}\label{eq43-18}
\|H^n\|&= \|\mathcal{A}^{\alpha,m}_{\tau}\eta^n_h-\sum_{j=m+1}^{n}b_{n-j} (P_{h}f(\pi_{h}^{1,0} u^j)-f(u^j))\|\lesssim h^{2}.
\end{aligned}\end{equation}
The error equation \eqref{eq43-17} is very similar to \eqref{error-eq}.
We can immediately obtain the convergence for the method  \eqref{FLMM-3}.

\begin{proof}
From \eqref{e4.3111} and \eqref{Arhon}, we can obtain
\begin{equation} \label{eq43-8}\begin{aligned}
\|\xi_h^n\|^2
\lesssim& \mathcal{A}^{-\alpha,m}_{\tau}
\Big(\|\mathcal{B}^{\alpha,m}R^n\|^2+\|H^n\|^2   +  \|\Phi^n\|^2\Big)\\
\lesssim& h^{4}+  \tau^{2\sigma_{m+1}-\alpha} \ell_n^{(\sigma_{m+1})}
+\tau^{\alpha}\ell_n^{(\sigma_{m})}\sum_{k=1}^m \|e^k/\tau^{\alpha}\|^2.
\end{aligned}\end{equation}
Applying $\|u_h^n-u(\cdot,t_n)\| \leq\|\eta_h^n\| + \|\xi_h^n\|$ and $\|\eta_h^n\|\lesssim h^{2}$
completes the proof.
\end{proof}


\section{Applications}\label{sec-5}
We present $\omega^{(\alpha)}(z)$ used in \eqref{Dtau}. We   discuss the use of the FBDF-1
that is also known as the Gr\"{u}nwald--Letnikov formula, the FBDF-2,
and the GNGF-2 to discretize the Caputo fractional derivative,
where the generating functions $\omega^{(\alpha)}(z)$ for these methods are shown
in Table \ref{s5:tb1}, while the  generating functions
$b(z)$ and $\hat{b}(z)$ are also displayed in Table \ref{s5:tb1}.

\begin{table}[!ht]
\caption{The generating functions $\omega^{(\alpha)}(z)$, $b(z)$, and $\hat{b}(z)$.}\label{s5:tb1}
\centering\footnotesize
\begin{tabular}{|c|c|c|c|c|c|c|c|c|c|c|c|c|}
\hline
&$\omega^{(\alpha)}(z)$& $b(z)$& $\hat{b}(z)$  \\ \hline
FBDF-1&$(1-z)^{\alpha}$&  1  & 1 \\
FBDF-2&$(\frac{3}{2}-2z+\frac{1}{2}z^2)^{\alpha}$&
$(3/2-z/2)^{-\alpha}$&$(3/2-z/2)^{-\alpha}$ \\
GNGF-2&$(1-z)^{\alpha}\left(1+\frac{\alpha}{2} - \frac{\alpha}{2} z\right)$&
$\left(1+\frac{\alpha}{2} - \frac{\alpha}{2} z\right)^{-1}$&
$\left(1+\frac{\alpha}{2} - \frac{\alpha}{2} z\right)^{-1}$ \\
\hline
\end{tabular}
\end{table}

In Examples \ref{eg5-1}--\ref{eg5-3}, we verify  that the assumptions \eqref{assumption-a}
and \eqref{assumption-c} hold for the FBDF-1, FBDF-2, and GNGF-2.

\begin{example}[FBDF-1]\label{eg5-1}
From Table \ref{s5:tb1}, it is very easy to verify that the
assumptions \eqref{assumption-a}--\eqref{assumption-c} hold if the FBDF-1 is used,
the details are omitted.
\end{example}

\begin{example}[FBDF-2]\label{eg5-2}
From Table \ref{s5:tb1}, it is easy to obtain
\begin{align*}
\hat{b}(z)=&b(z)=(3/2-z/2)^{-\alpha}
=(3/2)^{-\alpha}\sum_{n=0}^{\infty}3^{-n}a^{(-\alpha)}_nz^n,\\
b_n=&(3/2)^{-\alpha} 3^{-n}a^{(-\alpha)}_n \lesssim 3^{-n} \lesssim n^{-1-\alpha},n>0,\\
b_0=&(2/3)^{\alpha}> 1-(2/3)^{\alpha}=\sum_{n=1}^{\infty}\hat{b}_n.
\end{align*}
Hence, the assumption \eqref{assumption-a}  holds.
The proof  of   \eqref{assumption-c} is presented in \ref{sec-C}.
\end{example}

\begin{example}[GNGF-2]\label{eg5-3}
From Table \ref{s5:tb1}, it is easy to obtain
\begin{align*}
\hat{b}(z)=&b(z)=\left(1+\frac{\alpha}{2}- \frac{\alpha}{2} z\right)^{-1}
=\frac{2}{\alpha+2}\sum_{n=0}^{\infty} \left(\frac{\alpha}{2+\alpha}\right)^n z^n,\\
b_n=&\frac{2}{\alpha+2} \left(\frac{\alpha}{2+\alpha}\right)^n \lesssim n^{-1-\alpha},\ n>0,\\
b_0=&\frac{2}{\alpha+2}> 1-\frac{2}{\alpha+2}=\frac{\alpha}{\alpha+2}
=\sum_{n=1}^{\infty}\hat{b}_n,
\end{align*}
which verifies the assumption \eqref{assumption-a}.
The proof of  \eqref{assumption-c} is presented in  \ref{sec-C}.
\end{example}

Next, we show that the   BN-$\theta$ method in Ref. \refcite{Yinbaoli2020}
can be applied in the present framework.
The BN-$\theta$ method recovers   the  FBDF-2 (or GNGF-2) if $\theta=0$ (or $\theta=1/2$).
In the following example, we consider the BN-$\theta$ method for  $0\leq \theta \leq 1/2$.
\begin{example}[BN-$\theta$ method]\label{LY-BN_theta}
The generating functions $\omega^{(\alpha)}(z)$, $b(z)$, and $\hat{b}(z)$  are given by
\begin{align}
&\omega^{(\alpha)}(z)
=(1-z)^{\alpha} \left(\frac{3}{2}-\frac{z}{2} -\theta (1-z)\right)^{\alpha}
({1+\theta\alpha(1- z)}),\label{bntheta-w}\\
&b(z)=\hat{b}(z)=\frac{(1-z)^{\alpha}}{\omega^{(\alpha)}(z)}
=\frac{\left(\frac{3}{2}-\theta\right)^{-\alpha}}{1+\theta\alpha}
\frac{\left(1-\frac{1-2\theta}{3-2\theta}z\right)^{-\alpha}}
{1 -\frac{\theta\alpha}{1+\theta\alpha}z},\label{bntheta-b}
\end{align}
where $b_n$ can be expressed by
\begin{equation}\label{b-n}
b_n=b_0\sum_{j=0}^n a_{j}^{(-\alpha)}\left(\frac{1-2\theta}{3-2\theta}\right)^j
\left(\frac{\theta\alpha}{1+\theta\alpha}\right)^{n-j},\qquad
b_0=\frac{\left({3}/{2}-\theta\right)^{-\alpha}}{1+\theta\alpha}.
\end{equation}
Eq. \eqref{appendix-C-3} implies
$b_n\lesssim 2^{-n} \lesssim n^{-1-\alpha}$ and
$b_0 > \sum_{n=1}^{\infty}\hat{b}_n$ can be derived from
$$\frac{\left(1-\frac{1-2\theta}{3-2\theta}\right)^{-\alpha}}
{1 -\frac{\theta\alpha}{1+\theta\alpha}}
=(1+\theta\alpha)\left(\frac{3}{2}-\theta\right)^{\alpha}
\leq (1+ \theta)\left(\frac{3}{2}-\theta\right)\leq \left(\frac{1+3/2}{2}\right)^2<2,$$
which verifies \eqref{assumption-a}.
The proof of  \eqref{assumption-c} is presented in \ref{sec-C}.
\end{example}

In the rest of this section, we simply address that the two Crank-Nicolson (CN) type methods in Ref. \refcite{ZengLLT15}
 can be analyzed in the present frame work.
We do not show how to obtain the CN type methods, readers can refer to Ref. \refcite{ZengLLT15} for details.

The  CN  Galerkin FEM for  solving \eqref{e1.1}  reads as: Given $u^0_h=\pi_{h}^{1,0} u_0$,  find $u_h^n\in X_h$ for $n\geq 1$, such that
\begin{equation}\label{eq:zeng-sisc-13-5} \begin{aligned}
&  \frac{1}{\tau^{\alpha}}\sum_{j=1}^na_{n-j}^{(\alpha)}(u_h^j-u_h^0,v)
+\sum_{j=1}^nb_{n-j}(\nabla u_h^j,\nabla v)  \\
  =&\sum_{j=1}^nb_{n-j}\left(P_{h}f(u^{j}_h),v\right)
 +B_n\left(P_{h}f(u^{0}_h),v\right) - B_n(\nabla u_h^0,\nabla v),
\quad \forall v\in X_h,
\end{aligned} \end{equation}
where $B_n$ and $b_n$ are given by
\begin{eqnarray}
B_n& =& \frac{1}{\Gamma(1+\alpha)}\sum_{j=1}^na_{n-j}^{(\alpha)}j^{\alpha}-\sum_{j=0}^{n-1}b_j
=O(n^{-1}),\label{eq:zeng-sisc-13-4}\\
b(z)&=& 1-\frac{\alpha}{2} + \frac{\alpha}{2}z \quad \text{or}\quad b(z) = 2^{-\alpha}{(1+z)^{\alpha}}.\label{eq:bn-2}
\end{eqnarray}
If $\alpha\to1$, \eqref{eq:zeng-sisc-13-5} recovers the classical CN method. Obviously,
the  scheme  \eqref{eq:zeng-sisc-13-5}  is similar to \eqref{eq:zeng-5},
we can follow the convergence proof of \eqref{eq:zeng-5}
to prove the stability and convergence of \eqref{eq:zeng-sisc-13-5}
if the assumptions \eqref{assumption-a} and \eqref{assumption-c} hold.

Next, we verify the assumptions \eqref{assumption-a} and \eqref{assumption-c},
but we need to replace  $b(z)$ defined by \eqref{eq:omega} with \eqref{eq:bn-2}.
\begin{example}\label{eg5-4}
For  $b(z)=1-\frac{\alpha}{2} + \frac{\alpha}{2}z$,  we have  $\hat{b}(z)=b(z)$.
The assumption \eqref{assumption-a} follows from $b_0=1-\frac{\alpha}{2}\geq \frac{\alpha}{2}=\sum_{n=1}^{\infty}\hat{b}_n$.
For $n=0$, we have $c_0=b_0=1-\alpha/2>0$.
Using $a_{n-1}^{(-\alpha)}/a_{n}^{(-\alpha)}
= n /(n-1+\alpha)\leq \alpha^{-1}$ for
$n\geq 1$ yields
\begin{equation*}
c_n/a_n^{(-\alpha)}=2b_0-(b_0+b_1a_{n-1}^{(-\alpha)}/a_n^{(-\alpha)})
\geq (1-\alpha)/2 \geq 0,
\end{equation*}
which verifies  \eqref{assumption-c}.
\end{example}

\begin{example}\label{eg5-5}
For $b(z) = 2^{-\alpha}{(1+z)^{\alpha}}$,
we have $\hat{b}(z)= 2^{-\alpha}\left(2-(1-z)^{\alpha}\right).$
It is straightforward  to obtain
\begin{equation*}
b_n=2^{-\alpha}(-1)^na^{(\alpha)}_n,\qquad
b_0=2^{-\alpha}\geq 2^{-\alpha}=\sum_{n=1}^{\infty}\hat{b}_n,
\end{equation*}
which  verifies the assumption \eqref{assumption-a}. For $c_n$, we
have $c_0=b_0=2^{-\alpha}>0$ and $c_n= 0$ for $n>0$,
which verifies  \eqref{assumption-c}.

\end{example}
\section{Fast time-stepping methods}\label{sec-6}
We call   \eqref{FLMM-3} the direct method, which requires $O(n_T)$ memory and $O(n_T^2)$ computational cost in time.
In this section, we first present the fast version of \eqref{FLMM-3}, which significantly reduces the
memory  requirement and computational cost. Then, we propose a simple approach to  prove that the  fast method is convergent
as the direct method.


The basic idea   for fast calculating the discrete convolution
$\sum_{j=0}^n\omega_{n-j}^{(\alpha)}u^j$
is to represent the convolution weight  $\omega_{n}^{(\alpha)}$ as an integral (see Refs.
\refcite{BanjaiLopez18,GuoZeng19,LopLubSch08,SunNieDeng19} and \refcite{ZengTBK2018}).
We do not show how to derive the integral representation of $\omega_{n}^{(\alpha)}$,
this is not the main goal of this work, readers can refer Refs. \refcite{GuoZeng19} and \refcite{LopLubSch08}
for details.
We adopt the fast method in Ref. \refcite{GuoZeng19} for illustration, but
the fast methods in Refs. \refcite{BanjaiLopez18,LopLubSch08,SunNieDeng19} and \refcite{ZengTBK2018} can be applied in the present framework.

Due to $\sigma=0$ in (4.11) of Ref. \refcite{GuoZeng19}, the convolution weight $\omega_{n}^{(\alpha)}$ is expressed into
\begin{equation}\label{s6-0-1}
\omega_n^{(\alpha)} =\tau^{1+\alpha}
\int_{-\infty}^{\infty}(1+\tau e^x)^{-1-n} \phi(x) d,\qquad \phi(x) = -\frac{\sin(\alpha\pi)}{\pi}\frac{e^{(1+\alpha)x}}{b(-e^x)},
\end{equation}
where $b(z)$ is defined by \eqref{eq:omega}.
The above integral  can be approximated by the truncated trapezoidal rule  given by (see Ref. \refcite{GuoZeng19} (4.15))
\begin{equation}\label{s6-0-2}
\omega_n^{(\alpha)}=  \widetilde{\omega}_n^{(\alpha)}+O(n^{-\alpha-1}\epsilon),\qquad
\widetilde{\omega}_n^{(\alpha)}=\tau^{1+\alpha}\sum_{{\ell}=1}^{Q}\varpi_{\ell}(1+\tau  e^{\lambda_{\ell}})^{-1-n} ,
\quad n\geq n_0,
\end{equation}
where   $n_0$ is a suitable positive integer satisfying $n_0 > m+1$, the
quadrature point  $\lambda_{\ell}=x_{\min}+(\ell-1) \Delta x$, the quadrature weight
$\varpi_{\ell}=\Delta x \phi(x_{\ell})$,   $\Delta x = (x_{\max}-x_{\min})/Q$, $Q$ is the number of
quadrature points satisfying $Q\ll n_T$.
For a given precision $\epsilon$,  $x_{\max}$ and $x_{\min}$ are given by\cite{GuoZeng19}
$$
x_{\min }=\frac{\log (\epsilon)}{1+\alpha}-\log \left(n_{T} \tau\right), \quad x_{\max }=\log \left(\frac{-2 \log (\epsilon)+2(1+\alpha) \log \left(n_{0} \tau\right)}{n_{0} \tau}\right).
$$


With \eqref{s6-0-2}, we define the fast convolution quadrature  operator ${}_FD_{\tau}^{\alpha,m}$ as
\begin{equation}\label{s6-0-3-0}
{}_FD_{\tau}^{\alpha,m} u^n=
\frac{1}{\tau^{\alpha}}\sum_{j=n-n_{0}+1}^{n} \omega_{n-j}^{(\alpha)} (u^j-u^0)
+\frac{1}{\tau^{\alpha}}\sum_{j=1}^{n-n_{0}}  \widetilde{\omega}_{n-j}^{(\alpha)} (u^j-u^0)
+\frac{1}{\tau^{\alpha}}\sum_{j=1}^{m}w_{n,j}^{(m)} (u^{j}-u^0).
\end{equation}
Using \eqref{s6-0-2}, we  find that
${\tau^{-\alpha}}\sum_{j=1}^{n-n_{0}} \widetilde{\omega}_{n-j}^{(\alpha)} (u^j-u^0)$
in \eqref{s6-0-3-0} can be calculated by
\begin{equation}\label{s6-0-3}
\frac{1}{\tau^{\alpha}}\sum_{j=1}^{n-n_{0}}
 \widetilde{\omega}_{n-j}^{(\alpha)} (u^j-u^0)
= \sum_{{\ell}=1}^{Q}\varpi_{\ell}y^{n-n_0}_{\ell},
\end{equation}
where $y^{n}_{\ell}$ satisfies the following recurrence relation
\begin{equation}\label{s6-0-4}
y^{n}_{\ell}= \frac{1}{1+\tau e^{\lambda_{\ell}} }
\left[y^{n-1}_{\ell}+\tau (u^{n-1}-u^0)\right], \quad y^{0}_{\ell}=0.
\end{equation}
Clearly, the discrete convolution   $\frac{1}{\tau^{\alpha}}\sum_{j=1}^{n-n_{0}}  \widetilde{\omega}_{n-j}^{(\alpha)} (u^j-u^0)$
in \eqref{s6-0-3-0} is reformulated as  \eqref{s6-0-3}, which requires
$O(Q)$ storage and $O(Qn_T)$ computational cost.

We  replace $D_{\tau}^{\alpha,m} $ in \eqref{FLMM-3} with ${}_FD_{\tau}^{\alpha,m}$
to obtain the fast time-stepping Galerkin FEM for \eqref{e1.1} as:
Find ${}_Fu_h^n\in X_h$ for $n\geq n_0>m+1$, such that
\begin{equation}\label{s6-1}\left\{ \begin{aligned}
&({}_FD_{\tau}^{\alpha,m}{}_Fu_{h}^n,v)
+(\nabla {}_Fu_{h}^n,\nabla v)  =\left(P_{h}f({}_Fu_{h}^{n}),v\right),\quad \forall v\in X_h,\\
&{}_Fu_h^j=u_h^{j},\quad 0 \leq j \leq n_0-1,
\end{aligned} \right.\end{equation}
where $u_h^j$ is the solution of the direct method \eqref{FLMM-3} and
${}_FD_{\tau}^{\alpha,m}$ is defined by \eqref{s6-0-3}.

According to Refs. \refcite{GuoZeng19} and \refcite{Trefethen14}, $\widetilde{\omega}_n^{(\alpha)}$  can be expressed by
\begin{equation}\label{s6-0-5}
\widetilde{\omega}_n^{(\alpha)}  =(1+\varepsilon_n)\omega_n^{(\alpha)},
\end{equation}
where $\varepsilon_{n}$ is the  error that can be made arbitrarily small and
$\varepsilon_n=0$ for $0\leq n < n_0$.

We have the following theorem, the proof ow which is given in \ref{sec-E}.
\begin{theorem}\label{thm6-1}
Let $u_h^n$ and ${}_Fu_{h}^n$ be the solutions of \eqref{FLMM-3}  and
\eqref{s6-1}, respectively.
If the conditions in Theorem \ref{thm2-1} hold, $m\leq n_0$, and
$|\varepsilon_{n}| \lesssim \tau^{\alpha}\varepsilon$,   then
\begin{equation}\label{s6-2}
\|{}_Fu_{h}^n-u_h^n\|\lesssim \varepsilon.
\end{equation}
\end{theorem}

\section{Numerical results}\label{sec:numer}
In this section, we  perform numerical experiments to verify the efficiency of the
scheme \eqref{FLMM-3}. We focus on the following two aspects:
\begin{itemlist}
  \item Verify the accuracy and convergence of the scheme \eqref{FLMM-3} when
  the regularity of the analytical solution is known, i.e., $\delta_k$ is known.
  In such a case, the optimal choice of $\sigma_k$  should be
  $\sigma_k=\delta_k$; see Tables   \ref{table-11}--\ref{table-12}.
  \item If the regularity of the analytical solution is unknown, the method \eqref{FLMM-3}
  still works well by choosing suitable $\sigma_k$. For example, select $\sigma_k=k\alpha$ or
$\sigma_k \in \{\sigma_{\ell,j}|\sigma_{\ell,j}=\ell+j\alpha,\ell\in Z^+,j\in Z^+\}$,
accurate numerical solutions can still be obtained, the related numerical results are shown
in Tables \ref{table-13}--\ref{tbl-42}.
\end{itemlist}

At most four  correction terms are used  to achieve accurate  numerical solutions,
which  verifies that the present time-stepping \eqref{FLMM-3} is efficient.
This also demonstrates that Lubich's convolution quadrature with correction terms\cite{Lub86} is practically valuable.

\begin{example}\label{eg1}
Consider the  time-fractional subdiffusion equation
\begin{equation}\label{eg1:eq1}\left\{\begin{aligned}
&{}_{0}^{C} D_{t}^{\alpha}u=\frac{1}{2\pi^2}\Delta u + f(u), && (x,t)\in \Omega \times (0,T],T>0,\\
&u(x,y,0)= \sin(\pi x)\sin(\pi y),&& (x,y)\in \bar{\Omega},\\
&u(x,y,t)=0, &&(x,y,t)\in \partial\Omega\times  [0,T],
\end{aligned}\right.\end{equation}
where $\Delta = \partial_x^2 + \partial_y^2,(x,y)\in \Omega:=(0,1)^2$, and $0< \alpha \leq 1$.
\begin{itemlist}
  \item Case I: $f=0$, the exact solution of \eqref{eg1:eq1} is
  $u =E_{\alpha}(-t^{\alpha})\sin(\pi x)\sin(\pi y)$,
  where $E_{\alpha}(z)$ is the Mittag--Leffler function defined by
  $E_{\alpha}(z)=\sum_{k=0}^{\infty}\frac{z^k}{\Gamma(k\alpha+1)}$.
  \item Case II:  $f=u(1-u^2)$, the exact solution  of \eqref{eg1:eq1} is  unknown.
\end{itemlist}
\end{example}

We choose the FBDF-2 in time discretization, i.e.,
$\omega^{(\alpha)}(z)=(3/2-2z+z^2/2)^{\alpha}$ with $p=2$,
and the bicubic element in space approximation, i.e., $r=3$ in \eqref{Xh-r}.
The space step size is taken as $h=1/128$.
%
%
%
The error at $t=t_n$   is denoted by
$$e^n=u(t_n)-u_{h}^n.$$
If the  analytical solution is unavailable, then the reference solution is obtained
from the corresponding fast method \eqref{s6-1} with one correction  term and a smaller time stepsize $\tau=2^{-17}$.
The starting values used in \eqref{FLMM-3} for $m\geq 2$ is obtained by solving \eqref{FLMM-2} with
a  step size $10^{-1}\tau /\lceil\tau^{-1} \rceil$ and $m=1$.
We only show the accuracy in time.

For Case I, the exact solution is known, we have $\delta_k=k\alpha$,  so
$\sigma_k$ used in \eqref{FLMM-3} is chosen as $\sigma_k=k\alpha$.
By \eqref{eq:zeng-7-2-2}, the temporal error at $t=t_n\gg 0$ is
$O(\tau^2)$ for $m\alpha>1.5$,
$O(\tau^2\log(n))$ for $m\alpha=1.5$, and
$O(\tau^{m\alpha +0.5})$ for $m\alpha<1.5$.
%
%

Table \ref{table-11} shows the $L^2$ errors for $\alpha=0.2,0.5$ and $0.8$ at $t=1$.
For $\alpha=0.2$, the regularity of the analytical solution is low, but the accuracy of the
numerical solutions increases significantly as the number of correction terms increases up to four,
and the observed convergence rate is better than the theoretical result  $O(\tau^{m\alpha +0.5})$.
Second-order accuracy can be obtained if we increase $m$ and use the quadruple-precision in computations,
which seems  unnecessary for numerical  practices, since  four correction terms with double precision
can achieve sufficiently accurate numerical results.
As $\alpha$ increases, the regularity of the analytical solution improves,
three/two  correction terms are enough to achieve
second-order accuracy   for $\alpha=0.5$/$0.8$,
which agrees with the theoretical analysis.
However, better convergence rate is observed than the  theoretical convergence rate $O(\tau^{m\alpha +0.5})$
for $m\alpha<1.5$.
Table   \ref{table-12}  shows the maximum $L^2$ errors for  $\alpha=0.2,0.5$ and 0.8.
We observe that the accuracy of numerical solutions increases significantly as the number of the correction
terms increases, especially for a smaller fractional order $\alpha$, though the theoretical convergence
of the maximum $L^2$ error is $O(\tau^{(m+0.5)\alpha})$.
Both Tables \ref{table-11}  and \ref{table-12} demonstrate that a few number of corrections
are enough to achieve   accurate numerical solutions,
which will be further verified in Case II for solving nonlinear problems.

\begin{table}[!ht]
\caption{The  $L^2$ error  at $t=1$ for Case I, $\sigma_k=k\alpha,\tau=2^{-J}$.}\label{table-11}
\centering\footnotesize
 \begin{tabular}{|c|c|cc|cc|cc|cc|cc|}
\hline
$\alpha$&$J$  & $m = 0$    & rate & $m = 1$    & rate & $m = 2$    & rate & $m = 3$    & rate & $m = 4$    & rate   \\\hline
    &5&  3.98e-4 &      & 2.63e-5 &      & 4.33e-6 &      & 6.93e-7 &      & 7.08e-8 &           \\
    &6&  1.98e-4 & 1.00 & 1.19e-5 & 1.14 & 1.75e-6 & 1.31 & 2.42e-7 & 1.52 & 2.52e-8 & 1.49        \\
0.2 &7&  9.91e-5 & 1.00 & 5.37e-6 & 1.15 & 7.00e-7 & 1.32 & 8.51e-8 & 1.51 & 9.45e-9 & 1.42         \\
    &8&  4.95e-5 & 1.00 & 2.39e-6 & 1.17 & 2.78e-7 & 1.33 & 3.01e-8 & 1.50 & 3.54e-9 & 1.41         \\
    &9&  2.47e-5 & 1.00 & 1.06e-6 & 1.17 & 1.10e-7 & 1.34 & 1.07e-8 & 1.49 & 1.30e-9 & 1.45        \\\hline
    &5&  1.08e-3 &      & 1.16e-5 &      & 2.50e-6 &      & 8.09e-6 &      & 7.39e-6 &          \\
    &6&  5.36e-4 & 1.01 & 4.00e-6 & 1.54 & 8.39e-7 & 1.58 & 2.31e-6 & 1.81 & 2.61e-6 & 1.50       \\
0.5 &7&  2.67e-4 & 1.00 & 1.43e-6 & 1.49 & 2.51e-7 & 1.74 & 6.34e-7 & 1.87 & 8.21e-7 & 1.67        \\
    &8&  1.34e-4 & 1.00 & 5.13e-7 & 1.47 & 7.07e-8 & 1.83 & 1.69e-7 & 1.91 & 2.40e-7 & 1.77        \\
    &9&  6.67e-5 & 1.00 & 1.85e-7 & 1.47 & 1.92e-8 & 1.88 & 4.41e-8 & 1.94 & 6.70e-8 & 1.84       \\\hline
    &5&  2.03e-3 &      & 6.96e-5 &      & 2.72e-5 &      & 3.43e-5 &      & 1.90e-5 &          \\
    &6&  1.01e-3 & 1.01 & 1.98e-5 & 1.82 & 7.05e-6 & 1.95 & 9.50e-6 & 1.85 & 5.83e-6 & 1.70       \\
0.8 &7&  5.02e-4 & 1.01 & 5.59e-6 & 1.82 & 1.80e-6 & 1.97 & 2.51e-6 & 1.92 & 1.63e-6 & 1.84        \\
    &8&  2.50e-4 & 1.00 & 1.58e-6 & 1.82 & 4.58e-7 & 1.98 & 6.48e-7 & 1.95 & 4.34e-7 & 1.91        \\
    &9&  1.25e-4 & 1.00 & 4.46e-7 & 1.82 & 1.16e-7 & 1.99 & 1.65e-7 & 1.97 & 1.12e-7 & 1.95      \\
\hline
\end{tabular}
\end{table}

\begin{table}[!ht]
\caption{The maximum $L^2$ error $\max_{1\leq n \leq T/\tau}\|e^n\|$ for Case I, $\sigma_k=k\alpha,T=1,\tau=2^{-J}$.}\label{table-12}
\centering\footnotesize
 \begin{tabular}{|c|c|cc|cc|cc|cc|cc|}
\hline
$\alpha$&$J$  & $m = 0$    & rate & $m = 1$    & rate & $m = 2$    & rate & $m = 3$    & rate & $m = 4$    & rate   \\\hline
    &5&  2.07e-2 &      & 2.75e-4 &      & 2.35e-5 &      & 4.97e-6 &      & 5.14e-7 &           \\
    &6&  1.95e-2 & 0.09 & 2.36e-4 & 0.22 & 1.83e-5 & 0.36 & 3.56e-6 & 0.48 & 3.32e-7 & 0.63        \\
0.2 &7&  1.82e-2 & 0.10 & 2.00e-4 & 0.24 & 1.40e-5 & 0.38 & 2.49e-6 & 0.51 & 2.10e-7 & 0.66         \\
    &8&  1.69e-2 & 0.11 & 1.68e-4 & 0.25 & 1.06e-5 & 0.41 & 1.71e-6 & 0.54 & 1.29e-7 & 0.70         \\
    &9&  1.56e-2 & 0.12 & 1.40e-4 & 0.27 & 7.88e-6 & 0.43 & 1.16e-6 & 0.57 & 7.81e-8 & 0.73        \\\hline
    &5&  2.29e-2 &      & 1.46e-4 &      & 2.30e-5 &      & 1.21e-5 &      & 7.39e-6 &             \\
    &6&  1.71e-2 & 0.42 & 7.64e-5 & 0.94 & 1.08e-5 & 1.10 & 4.28e-6 & 1.50 & 2.61e-6 & 1.50          \\
0.5 &7&  1.26e-2 & 0.44 & 3.89e-5 & 0.98 & 4.66e-6 & 1.21 & 1.41e-6 & 1.60 & 8.21e-7 & 1.67           \\
    &8&  9.12e-3 & 0.46 & 1.95e-5 & 1.00 & 1.90e-6 & 1.29 & 4.43e-7 & 1.67 & 2.40e-7 & 1.77           \\
    &9&  6.58e-3 & 0.47 & 9.70e-6 & 1.01 & 7.45e-7 & 1.35 & 1.33e-7 & 1.73 & 6.70e-8 & 1.84         \\\hline
    &5&  1.08e-2 &      & 1.16e-4 &      & 3.76e-5 &      & 3.43e-5 &      & 1.90e-5 &            \\
    &6&  6.43e-3 & 0.75 & 3.76e-5 & 1.62 & 1.04e-5 & 1.85 & 9.50e-6 & 1.85 & 5.83e-6 & 1.70         \\
0.8 &7&  3.77e-3 & 0.77 & 1.21e-5 & 1.63 & 2.79e-6 & 1.91 & 2.51e-6 & 1.92 & 1.63e-6 & 1.84          \\
    &8&  2.19e-3 & 0.78 & 3.93e-6 & 1.63 & 7.26e-7 & 1.94 & 6.48e-7 & 1.95 & 4.34e-7 & 1.91          \\
    &9&  1.26e-3 & 0.79 & 1.27e-6 & 1.62 & 1.86e-7 & 1.96 & 1.65e-7 & 1.97 & 1.12e-7 & 1.95         \\
\hline
\end{tabular}
\end{table}

Next, we numerically display how $\sigma_k$ influence the
accuracy of numerical solutions when
$\sigma_k \notin\{\delta_1,\delta_2,\cdots,\delta_m,\cdots\}$. In such a case,
the discretization error in time is $O(\tau^{0.5}t_n^{(\alpha-1)/2})$  by Theorem \ref{thm2-1}.
We   consider Case I and take $\alpha=0.2$ and $\sigma_k=10^{-1}(2k-1)$ in numerical simulations.
Table  \ref{table-13} displays the $L^2$ errors  at $t=1$, where  we  observe about first-order accuracy for all $0\leq m \leq 4$.
What is interesting  is that the error still decreases significantly as $m$ increases, though the
convergence rate  is almost not improved.
Similar results in Table \ref{table-14} are observed,  where the maximum $L^2$ errors are displayed.
This phenomenon was studied   in Ref. \refcite{ZengZK17}, which could be simply explained from
the fact that the time discretization \eqref{TD-1} (see also \eqref{e1.3}) is exact for $u=t^{\sigma_k}$.
Since $\sigma_k \notin\{\delta_1,\delta_2,\delta_3,\cdots\}$,
the   leading  term of the time discretization error, which we denote as
 $R_{0}^n(\delta_1)
=R_{0}^n(\sigma_1,\sigma_2,\cdots,\sigma_m,\delta_1)$, depends on
$t^{\delta_1}$ (or  $\delta_1$).
From  \eqref{cond-2}, one knows that $R_0^n(\delta)=0$
for all $\delta\in\{\sigma_1,\sigma_2,\cdots,\sigma_m\}$  and
$R_0^n(\delta)$ is an  analytical function with respect to $\delta$.  Hence,
it is reasonable to believe that
$|R_0^n(\delta)|$ contains the factor $S_m(\delta)=\prod_{k=1}^m |\delta-\sigma_k|$
and $S_m(\delta)$ may be small.
For Case I, $\alpha=0.2$ and $\sigma_k=10^{-1}(2k-1)$,  one has $\delta_1=0.2$,
$S_1(\delta_1)=10^{-1},S_2(\delta_1)=10^{-2},S_3(\delta_1)=3\times 10^{-3}$, and $S_4(\delta_1)=1.5\times 10^{-3}$.
From Tables \ref{table-13}--\ref{table-14}, we indeed observe that
the accuracy increases   as  $S_m(\delta_1)$  decreases.
Even if the regularity of the analytical solution is unknown,
adding suitable correction terms may help improve the accuracy of the numerical solutions; see
also  related results in Tables \ref{tbl-21}--\ref{tbl-22}
 and Tables 3 and 10, and Fig. 2.2. of Ref. \refcite{ZengZK17}.

\begin{table}[!ht]
\caption{The $L^2$ error $\|e^n\|$ at $t=1$ for Case I, $\alpha=0.2,\sigma_k=(2k-1)/10$.}\label{table-13}
\centering\footnotesize
 \begin{tabular}{|c|cc|cc|cc|cc|cc|}
\hline
$1/\tau$  & $m = 0$    & rate & $m = 1$    & rate & $m = 2$    & rate & $m = 3$    & rate & $m = 4$    & rate   \\\hline
32  &  3.98e-4 &      & 1.97e-5 &      & 8.75e-6 &      & 2.81e-6 &      & 9.63e-7 &           \\
64  &  1.98e-4 & 1.00 & 1.12e-5 & 0.82 & 4.37e-6 & 1.00 & 1.32e-6 & 1.09 & 4.68e-7 & 1.04       \\
128 &  9.91e-5 & 1.00 & 6.19e-6 & 0.85 & 2.17e-6 & 1.01 & 6.26e-7 & 1.08 & 2.27e-7 & 1.05        \\
256 &  4.95e-5 & 1.00 & 3.37e-6 & 0.88 & 1.07e-6 & 1.02 & 2.97e-7 & 1.08 & 1.09e-7 & 1.06        \\
512 &  2.47e-5 & 1.00 & 1.80e-6 & 0.90 & 5.23e-7 & 1.03 & 1.41e-7 & 1.07 & 5.17e-8 & 1.07      \\\hline
\end{tabular}
\end{table}

\begin{table}[!ht]
\caption{The maximum $L^2$ error $\max_{1\leq n \leq T/\tau}\|e^n\|$ for Case I, $\alpha=0.2,\sigma_k=(2k-1)/10$.}\label{table-14}
\centering\footnotesize
 \begin{tabular}{|c|cc|cc|cc|cc|cc|}
\hline
$1/\tau$  & $m = 0$    & rate & $m = 1$    & rate & $m = 2$    & rate & $m = 3$    & rate & $m = 4$    & rate   \\\hline
32  &  2.07e-2 &      & 1.88e-4 &       & 5.82e-5 &      & 1.89e-5 &      & 4.60e-6 &           \\
64  &  1.95e-2 & 0.09 & 2.04e-4 & -0.12 & 5.57e-5 & 0.06 & 1.74e-5 & 0.12 & 4.10e-6 & 0.17      \\
128 &  1.82e-2 & 0.10 & 2.15e-4 & -0.07 & 5.25e-5 & 0.09 & 1.57e-5 & 0.14 & 3.61e-6 & 0.18       \\
256 &  1.69e-2 & 0.11 & 2.20e-4 & -0.04 & 4.88e-5 & 0.10 & 1.41e-5 & 0.16 & 3.17e-6 & 0.19       \\
512 &  1.56e-2 & 0.12 & 2.21e-4 & -0.01 & 4.51e-5 & 0.12 & 1.25e-5 & 0.17 & 2.77e-6 & 0.20     \\\hline
\end{tabular}
\end{table}

For Case II,  we know $\delta_1=\alpha$, but we do not exactly know  $\delta_k$   for $k\ge2$.
Based on the criteria on selecting $\sigma_k$ (see lines below \eqref{eq:eq-w}), we take $\sigma_k=k\alpha$  in numerical simulations.

Take $\alpha=0.2,$
the $L^2$ errors at $t=1$ and the maximum $L^2$ errors are
displayed in Tables  \ref{tbl-21} and \ref{tbl-22},  respectively.
We can see that the accuracy is improved significantly as $m$ increases,
though the regularity of the solution is unknown.

\begin{table}[!ht]
\caption{The $L^2$ error $\|e^n\|$ at $t=1$ for Case II,   $\alpha=0.2,\sigma_k=k\alpha,k\le 4$.}\label{tbl-21}
\centering\footnotesize
 \begin{tabular}{|c|cc|cc|cc|cc|cc|}
\hline
$1/\tau$  & $m = 0$    & rate & $m = 1$    & rate & $m = 2$    & rate & $m = 3$    & rate & $m = 4$    & rate   \\\hline
32   & 2.01e-4 &      & 1.48e-5 &      & 3.20e-6 &      & 7.56e-7 &      & 1.69e-7 &      \\
64   & 1.00e-4 & 1.00 & 6.93e-6 & 1.10 & 1.39e-6 & 1.21 & 3.02e-7 & 1.33 & 7.35e-8 & 1.20      \\
128  & 5.00e-5 & 1.00 & 3.20e-6 & 1.11 & 5.94e-7 & 1.22 & 1.21e-7 & 1.31 & 3.23e-8 & 1.19       \\
256  & 2.50e-5 & 1.00 & 1.47e-6 & 1.12 & 2.53e-7 & 1.23 & 4.90e-8 & 1.31 & 1.39e-8 & 1.22       \\
512  & 1.25e-5 & 1.00 & 6.70e-7 & 1.13 & 1.07e-7 & 1.24 & 1.96e-8 & 1.32 & 5.66e-9 & 1.30    \\
\hline
\end{tabular}
\end{table}

\begin{table}[!ht]
\caption{The maximum $L^2$ error  $\max_{1\leq n \leq T/\tau}\|e^n\|$ for Case II,  $\alpha=0.2,\sigma_k=k\alpha,k\le 4,T=1$.}\label{tbl-22}
\centering\footnotesize
 \begin{tabular}{|c|cc|cc|cc|cc|cc|}
\hline
$1/\tau$  & $m = 0$    & rate & $m = 1$    & rate & $m = 2$    & rate & $m = 3$    & rate & $m = 4$    & rate   \\\hline
32   & 9.58e-3 &      & 1.43e-4 &      & 1.77e-5 &      & 5.52e-6 &      & 1.09e-6 &       \\
64   & 9.01e-3 & 0.09 & 1.26e-4 & 0.18 & 1.50e-5 & 0.24 & 4.61e-6 & 0.26 & 8.94e-7 & 0.28       \\
128  & 8.44e-3 & 0.10 & 1.11e-4 & 0.19 & 1.26e-5 & 0.24 & 3.85e-6 & 0.26 & 7.33e-7 & 0.29        \\
256  & 7.86e-3 & 0.10 & 9.69e-5 & 0.20 & 1.06e-5 & 0.25 & 3.21e-6 & 0.26 & 5.90e-7 & 0.31        \\
512  & 7.29e-3 & 0.11 & 8.42e-5 & 0.20 & 8.93e-6 & 0.25 & 2.64e-6 & 0.28 & 4.53e-7 & 0.38     \\
\hline
\end{tabular}
\end{table}

Table \ref{tbl-41} displays the $L^2$ errors at $t=1$ for $\alpha=0.8$, the accuracy increases as $m$ increases up to two,
about second-order accuracy is observed when $m=2$.
For $m\geq 3$,  the accuracy decreases as $m$ increases, which could be explained from \eqref{eq:zeng-7-2},
where the error  $\mathcal{E}^n$ induced by the starting values
dominates  the overall accuracy  and  increases as $m$ increases when $m\geq 2$; see Remark \ref{remark-2}.
Direct computation
shows $\ell_n^{(\sigma_1)}=\ell_n^{(\sigma_2)}=n^{-0.2}$,
$\ell_n^{(\sigma_3)}=1$, and $\ell_n^{(\sigma_4)}=n^{1.6}$. The negative effect caused by
$\ell_n^{(\sigma_m)}$ for $m=3,4$ is observed in Table \ref{tbl-41}.


We also take $\alpha=0.8$, but select $\sigma_k \in \{\sigma_{\ell,j}|\sigma_{\ell,j}=\ell+j\alpha,\ell\in Z^+,j\in Z^+\}$
in numerical simulations, i.e., $\sigma_1=0.8,\sigma_2=1,\sigma_3=1.6$ and $\sigma_4=1.8$.
From Table \ref{tbl-42}, we can see that second-order accuracy is observed for $m\geq 3$.
Although we cannot claim that the solution contains $t$, $t^{1.6}$, and $t^{1.8}$, what we
observe is that the selected $\sigma_k$  can
help  to improve the accuracy of numerical solutions.

We find that the analytical solution possibly contains the term $t^{2\alpha}$ (see $m=2$ in Table \ref{tbl-41} and
$m=3$ in Table \ref{tbl-42}),
since the accuracy is improved significantly when $t^{2\alpha}$
is exactly calculated in the numerical method.

\begin{table}[!ht]
\caption{The $L^2$ error $\|e^n\|$ at $t=1$ for Case II,   $\alpha=0.8,\sigma_k=k\alpha$.}\label{tbl-41}
\centering\footnotesize
 \begin{tabular}{|c|cc|cc|cc|cc|cc|}
\hline
$1/\tau$  & $m = 0$    & rate & $m = 1$    & rate & $m = 2$    & rate & $m = 3$    & rate & $m = 4$    & rate   \\\hline
32   & 1.02e-3 &      & 5.60e-5 &      & 1.75e-5 &      & 1.71e-4 &      & 7.36e-4 &      \\
64   & 5.07e-4 & 1.01 & 1.64e-5 & 1.77 & 4.74e-6 & 1.88 & 5.81e-5 & 1.56 & 3.27e-4 & 1.17  \\
128  & 2.52e-4 & 1.01 & 4.74e-6 & 1.79 & 1.26e-6 & 1.91 & 1.76e-5 & 1.72 & 1.18e-4 & 1.47  \\
256  & 1.26e-4 & 1.00 & 1.36e-6 & 1.80 & 3.32e-7 & 1.92 & 4.94e-6 & 1.83 & 3.69e-5 & 1.67  \\
512  & 6.29e-5 & 1.00 & 3.90e-7 & 1.80 & 8.90e-8 & 1.90 & 1.32e-6 & 1.90 & 1.06e-5 & 1.80
   \\\hline
\end{tabular}
\end{table}

\begin{table}[!ht]
\caption{The $L^2$ error $\|e^n\|$ at $t=1$ for Case II,   $\alpha=0.8,\sigma_1=0.8,
\sigma_2=1,\sigma_3=1.6,\sigma_4=1.8$.}\label{tbl-42}
\centering\footnotesize
 \begin{tabular}{|c|cc|cc|cc|cc|cc|}
\hline
$1/\tau$  & $m = 0$    & rate & $m = 1$    & rate & $m = 2$    & rate & $m = 3$    & rate & $m = 4$    & rate   \\\hline
32   & 1.02e-3 &      & 5.60e-5 &      & 1.59e-5 &      & 1.50e-5 &      & 2.29e-5 &       \\
64   & 5.07e-4 & 1.01 & 1.64e-5 & 1.77 & 5.45e-6 & 1.54 & 4.48e-6 & 1.74 & 6.54e-6 & 1.81       \\
128  & 2.52e-4 & 1.01 & 4.73e-6 & 1.79 & 1.74e-6 & 1.65 & 1.24e-6 & 1.85 & 1.46e-6 & 2.16        \\
256  & 1.26e-4 & 1.00 & 1.35e-6 & 1.81 & 5.28e-7 & 1.72 & 3.27e-7 & 1.92 & 2.82e-7 & 2.37         \\
512  & 6.29e-5 & 1.00 & 3.83e-7 & 1.82 & 1.54e-7 & 1.78 & 8.36e-8 & 1.97 & 5.53e-8 & 2.35
   \\\hline
\end{tabular}
\end{table}

Finally, we display  the numerical solutions for Case II   at $t=1,5,10,20$,
see Figure \ref{fig2}.  For the selected computational domain and initial data, we observe that
the solution decays as time $t$ evolves and it decays faster as the fractional order $\alpha$ increases.

\begin{figure}[htbp]
\centering
\subfigure{
\begin{minipage}[t]{0.25\linewidth}
\centering
\includegraphics[width=1.2in]{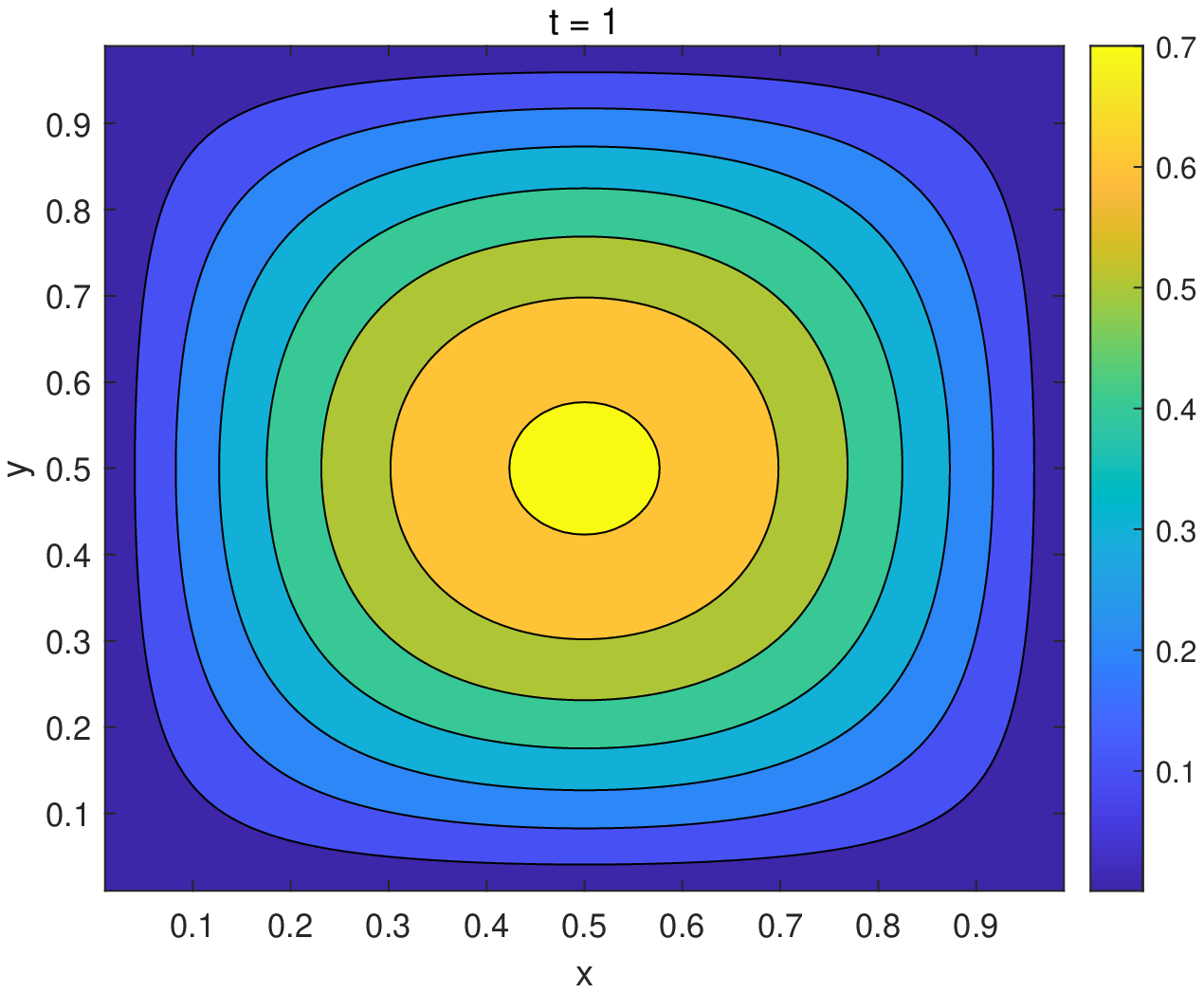}
\includegraphics[width=1.2in]{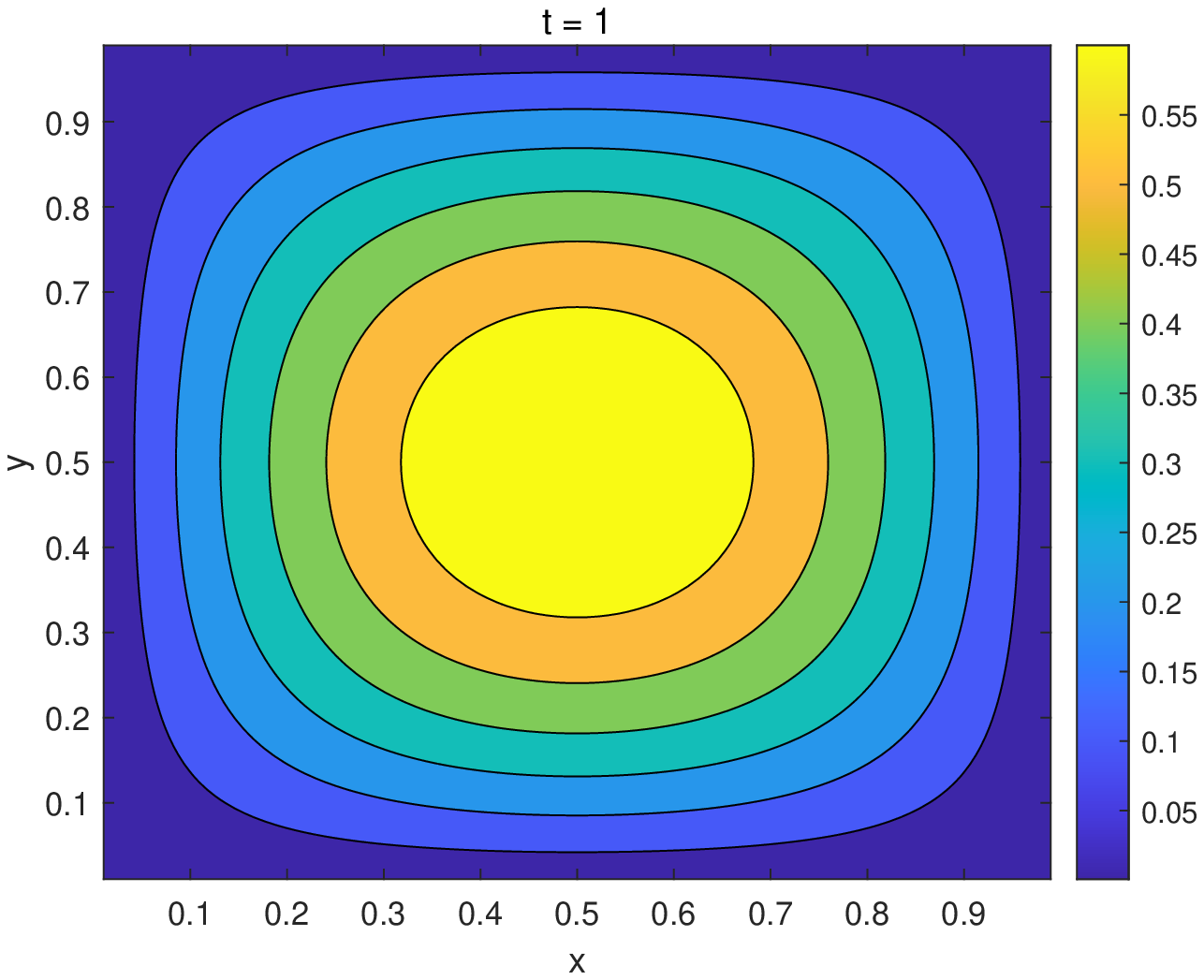}
\includegraphics[width=1.2in]{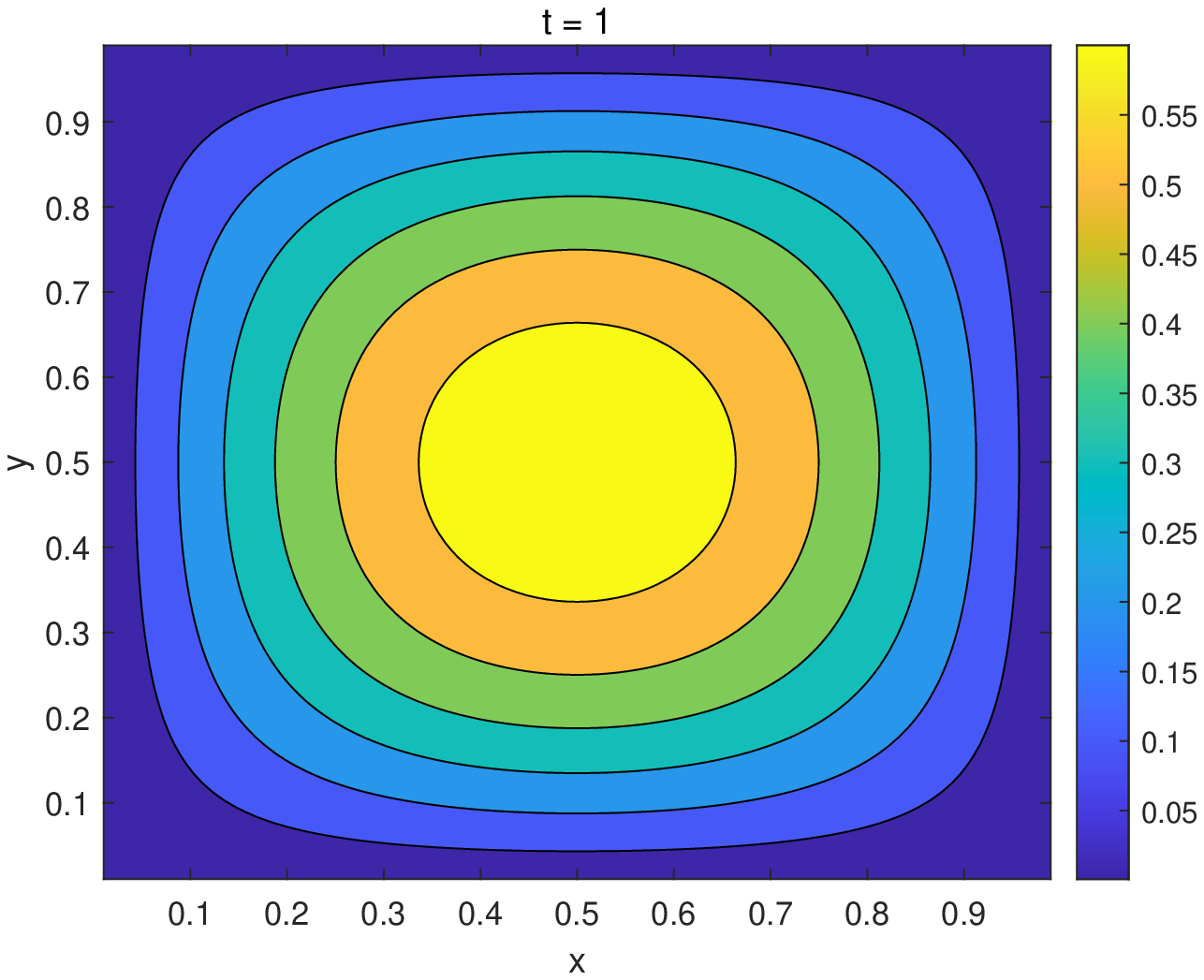}
\end{minipage}%
}%
\subfigure{
\begin{minipage}[t]{0.25\linewidth}
\centering
\includegraphics[width=1.2in]{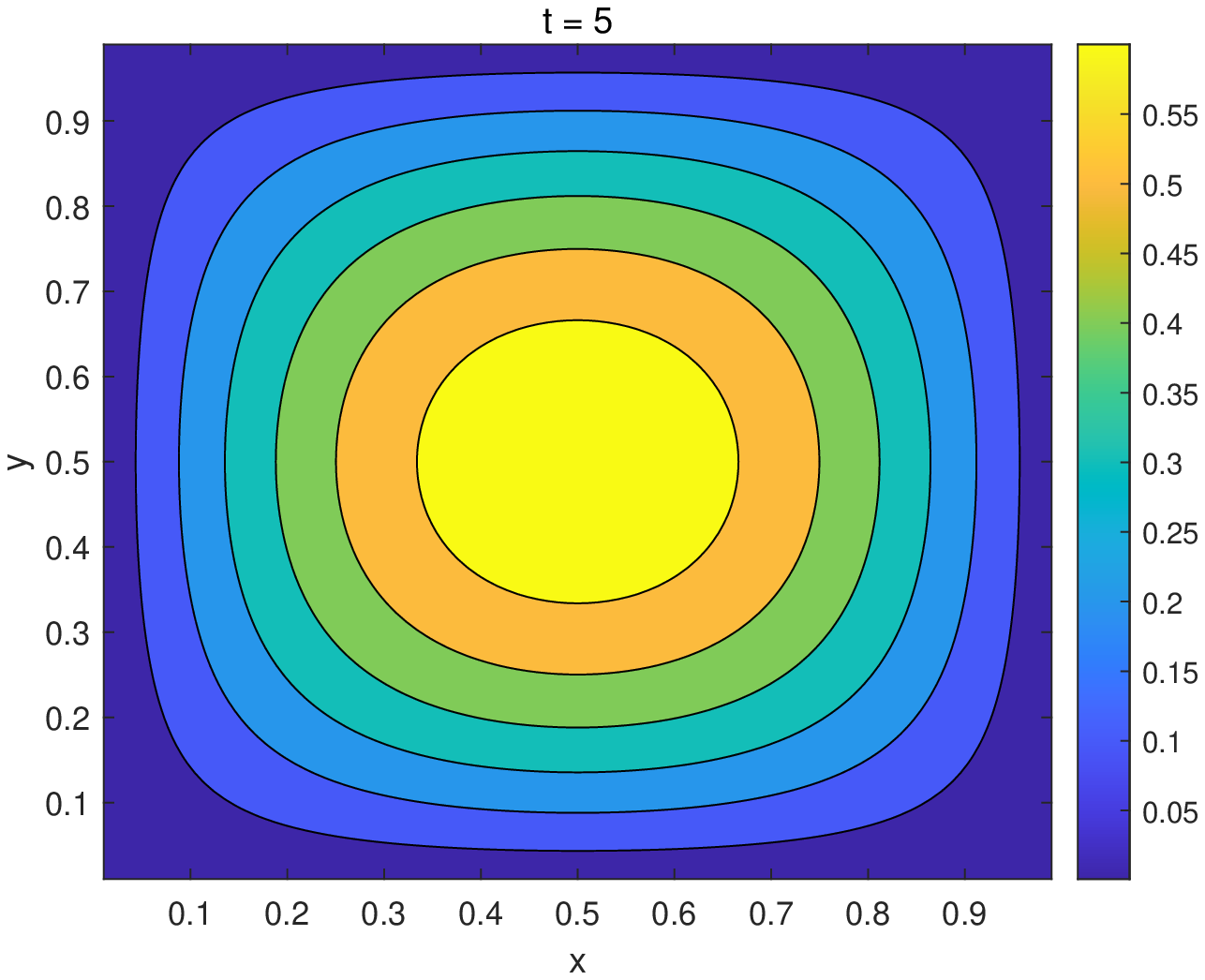}
\includegraphics[width=1.2in]{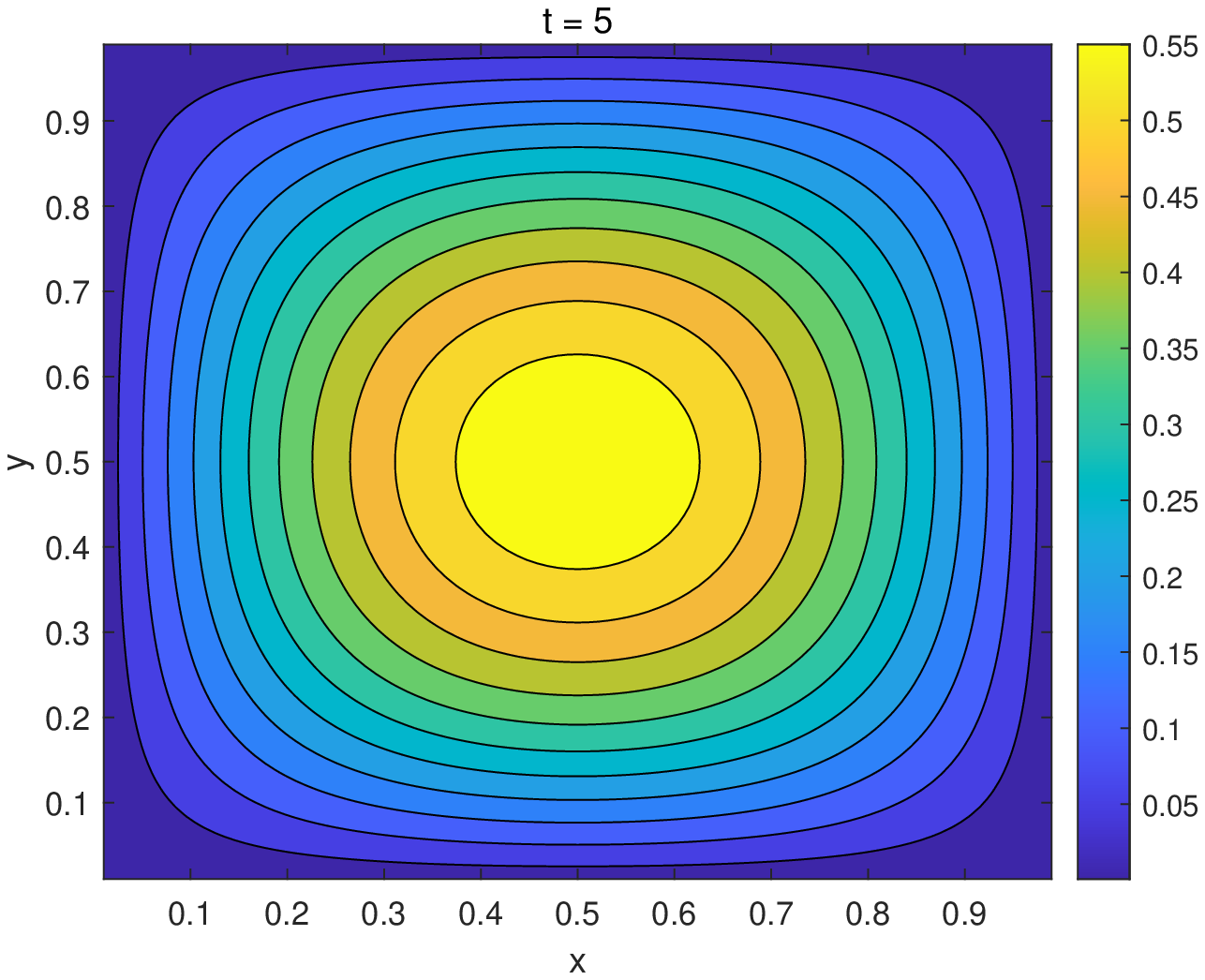}
\includegraphics[width=1.2in]{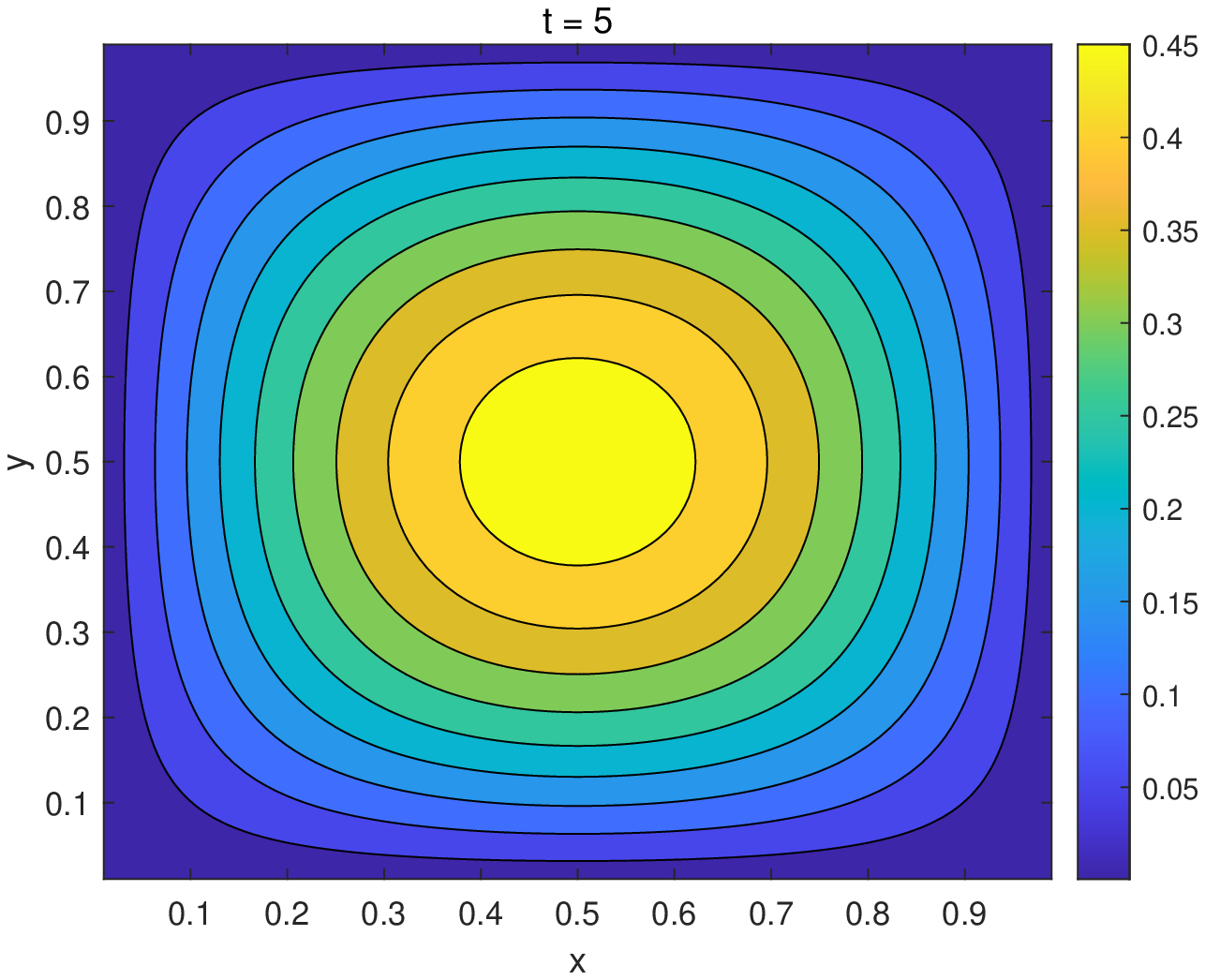}
\end{minipage}%
}%
\subfigure{
\begin{minipage}[t]{0.25\linewidth}
\centering
\includegraphics[width=1.2in]{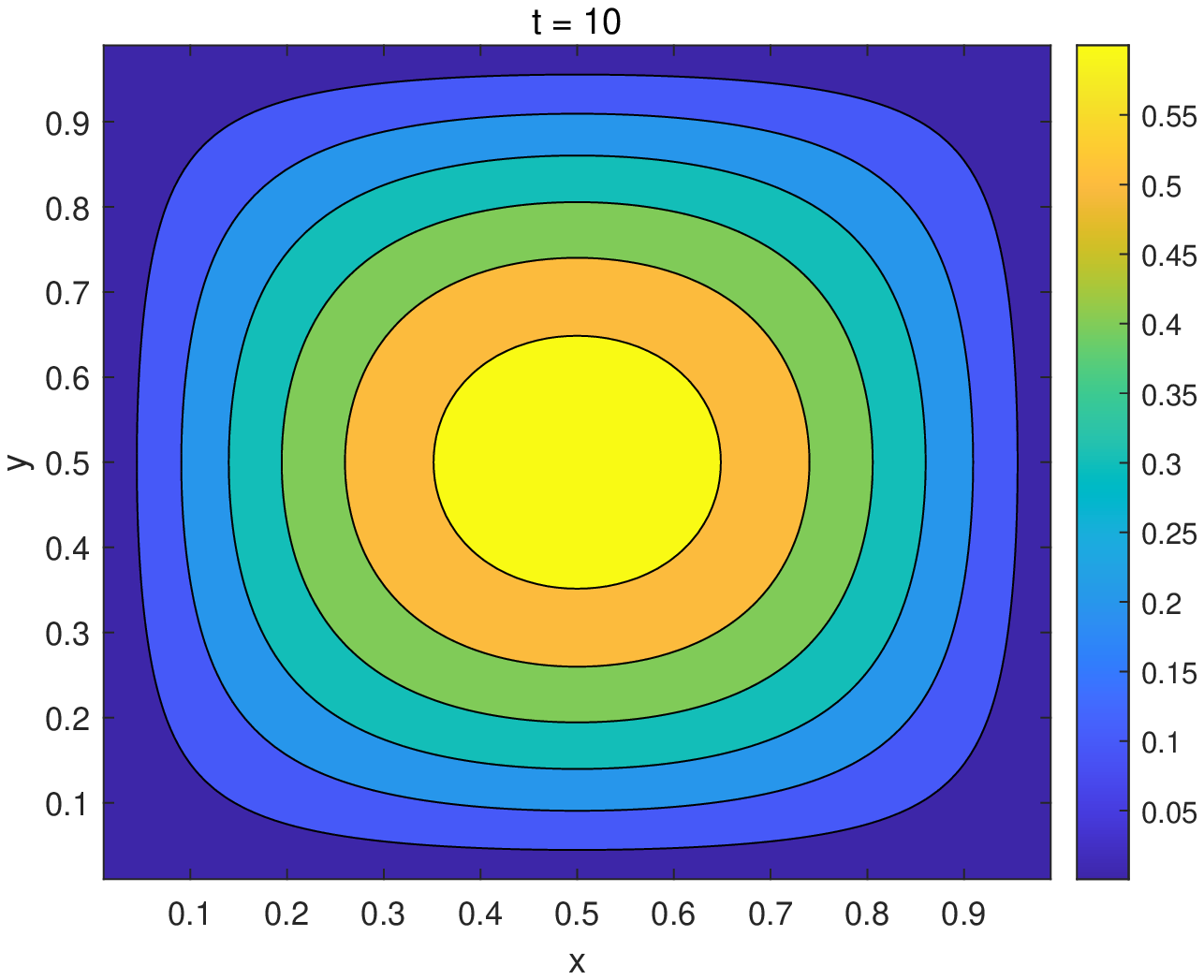}
\includegraphics[width=1.2in]{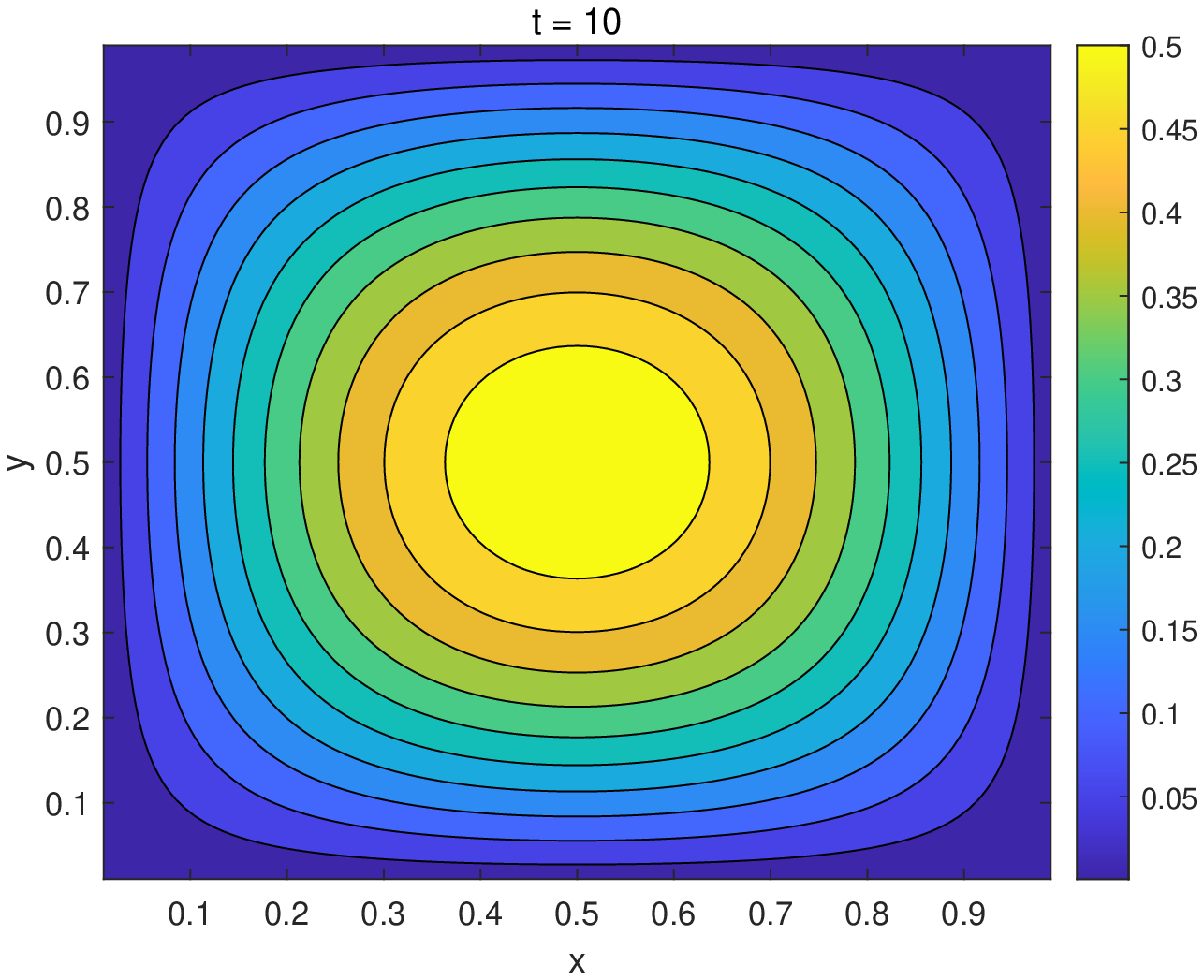}
\includegraphics[width=1.2in]{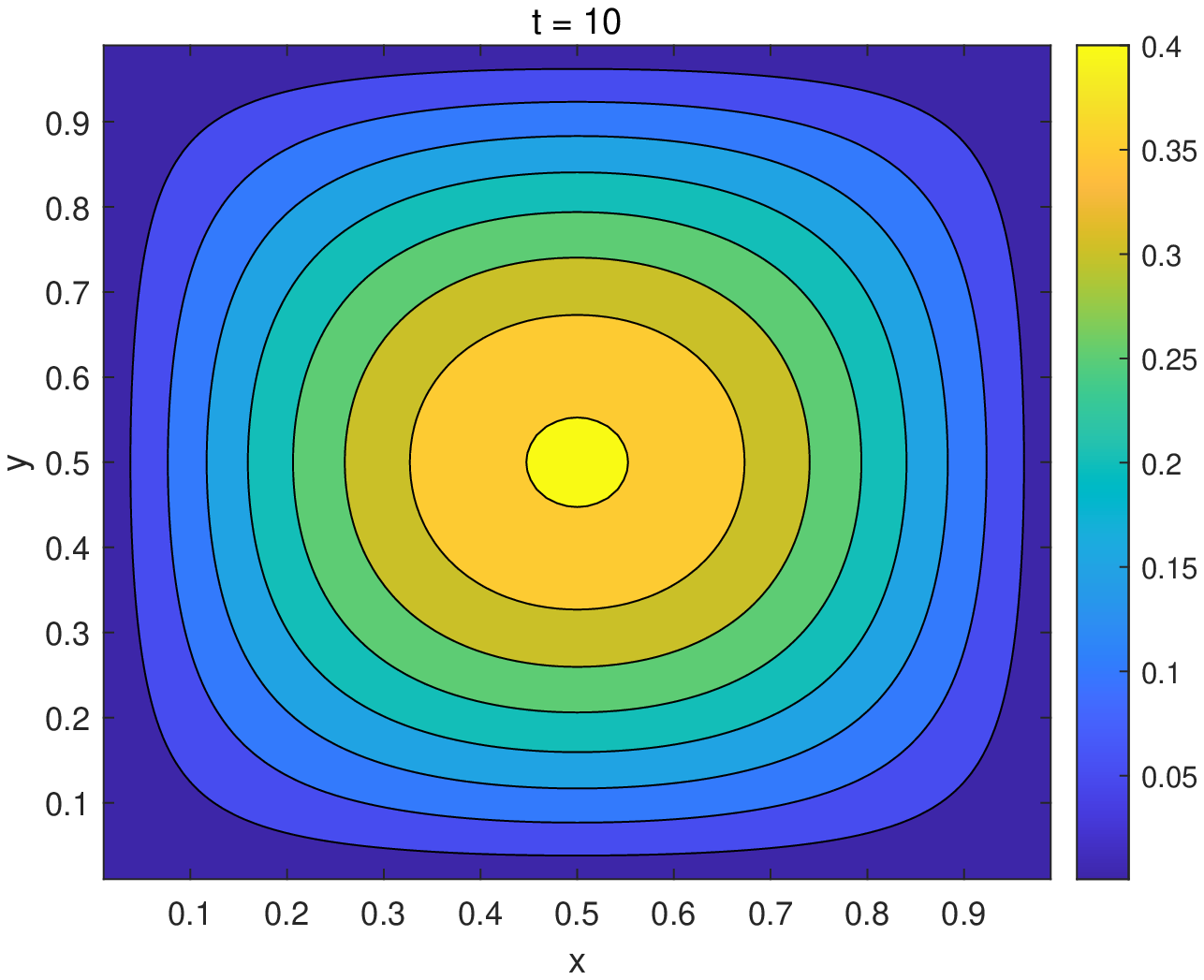}
\end{minipage}
}%
\subfigure{
\begin{minipage}[t]{0.25\linewidth}
\centering
\includegraphics[width=1.2in]{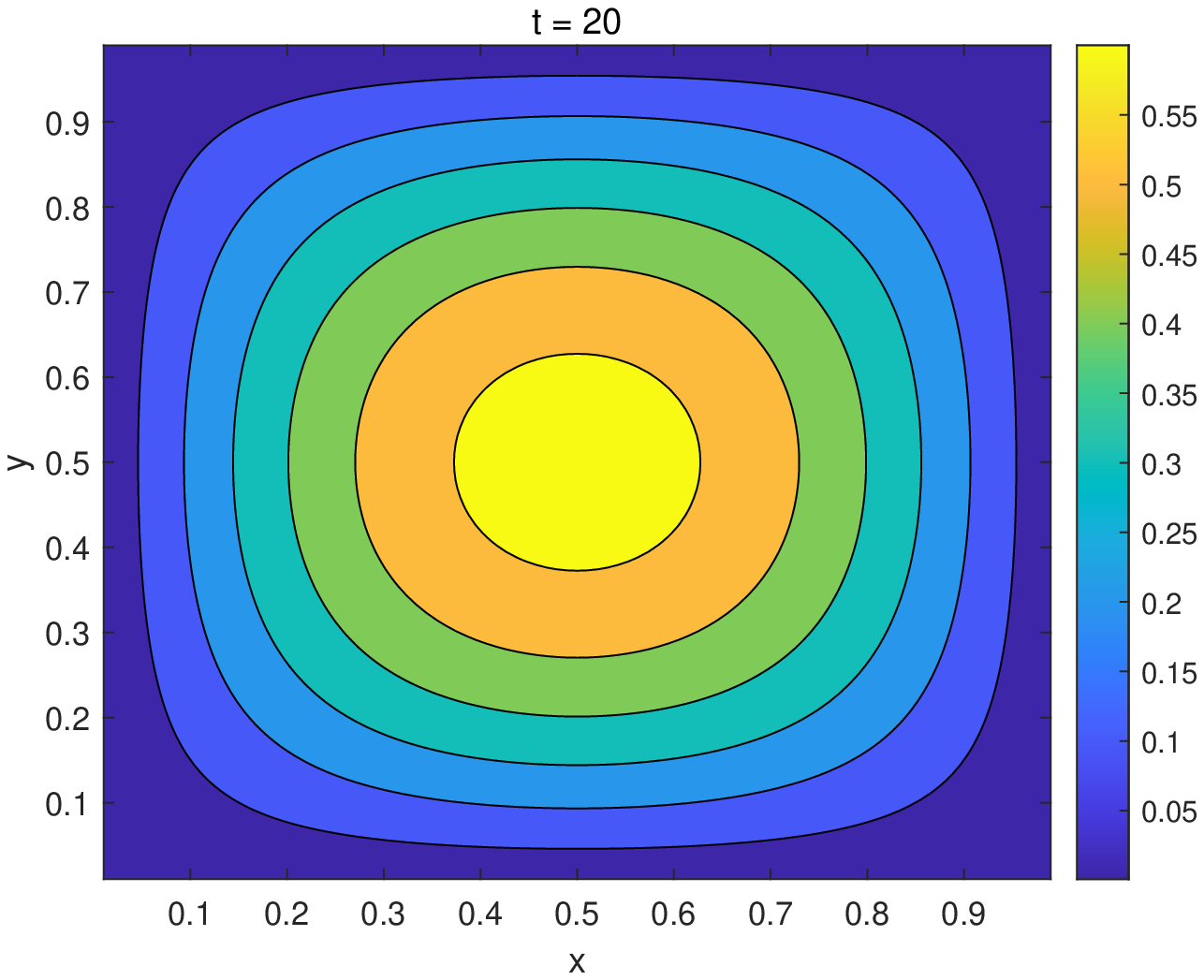}
\includegraphics[width=1.2in]{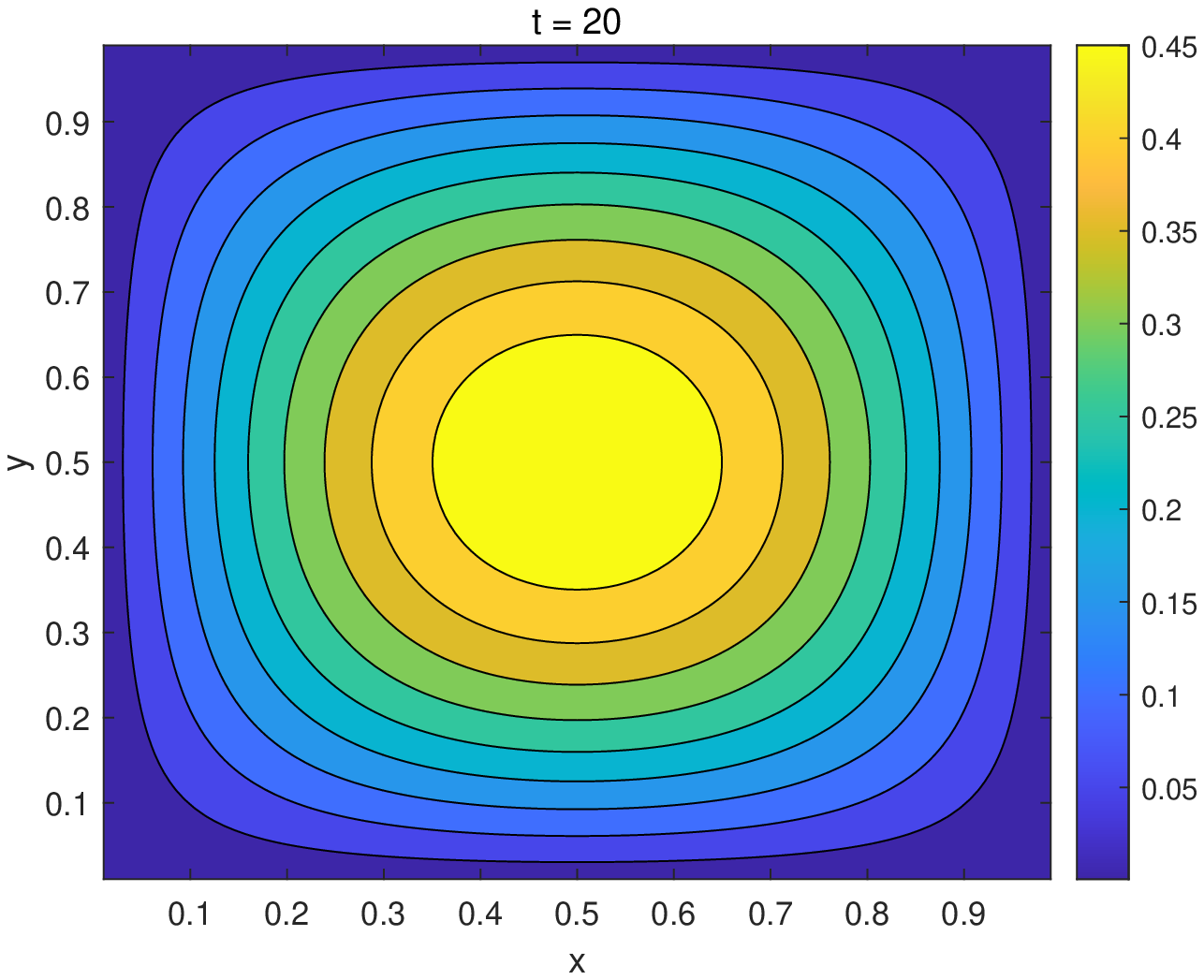}
\includegraphics[width=1.2in]{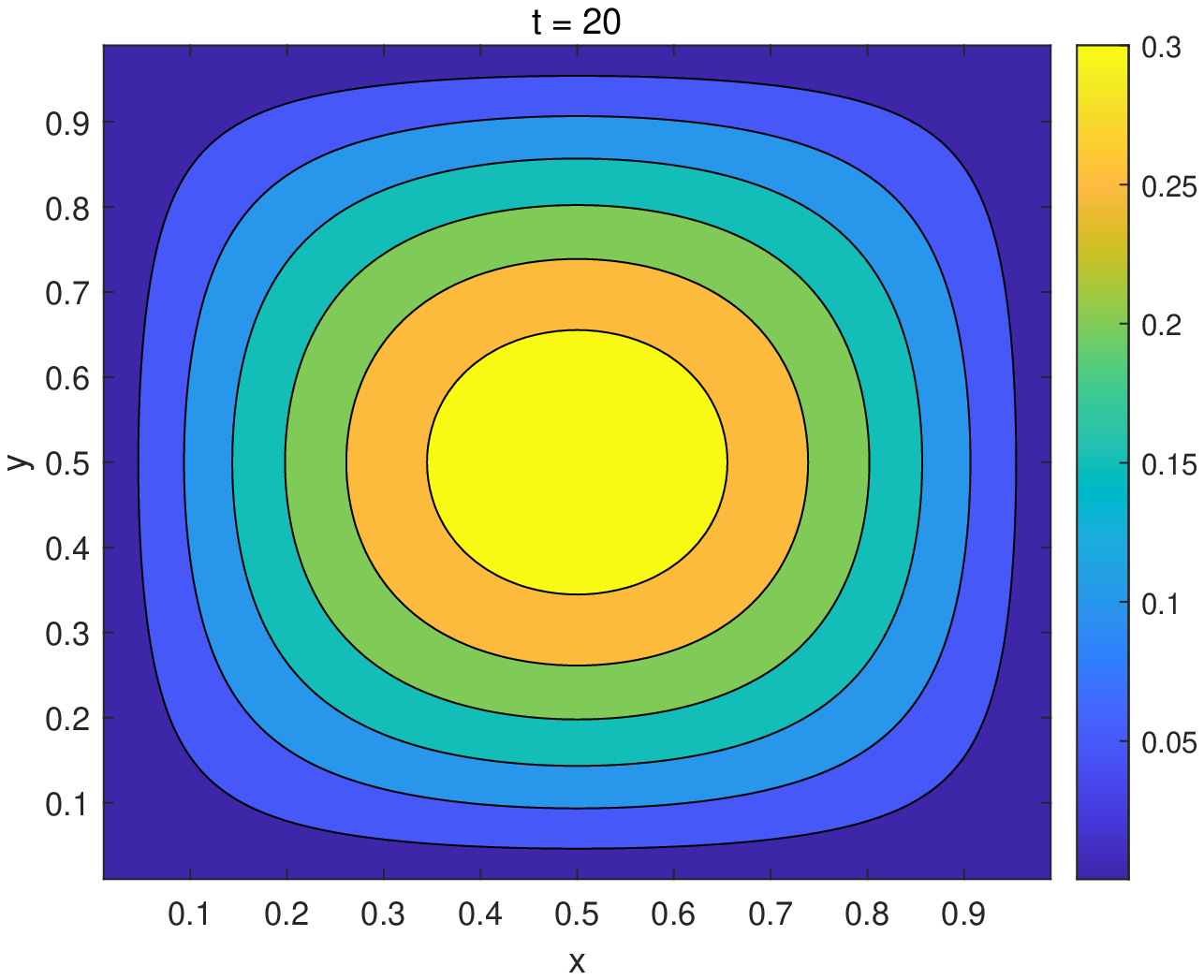}
\end{minipage}
}%
\centering
\caption{Numerical solutions for Case II, $\tau=0.01,h=1/32,m=1$,
 $\alpha=0.2,0.5,0.8$ (from top to bottom),  and $t=1,5,10,20$ (from left to right).  \label{fig2}}
\end{figure}

\section{Conclusion and discussion} \label{sec7}
In this paper, we show how to apply the generalized discrete Gr\"{o}nwall's inequality to
prove the convergence of a class of fully implicit time-stepping Galerkin FEM
for the one-dimensional  nonlinear subdiffusion equations. The
correction terms are used to deal with the initial singularity  of the solution.
The convergence analysis for this kind of time-stepping schemes
is limited, hence this work provides a simple approach to the convergence
of the time-stepping schemes with correction terms.
We also show a simple way to prove the convergence of  the fast time-stepping
Galerkin FEM based on the convergence of the direct time-stepping schemes.
It is hopeful that the methodology used in the convergence analysis of the present fast
method can be extended to simplify the  convergence analysis in Refs. \refcite{JiangZZZ17} and \refcite{ZhuXu2019}.

If the nonlinear term $f(u^n)$  is approximated by the first-order
extrapolation $f(u^{n-1})$ or second-order   extrapolation $2f(u^{n-1})-f(u^{n-2})$,
then we obtain the  semi-implicit time-stepping FEMs, the convergence of which
can be obtained directly.
The convergence analysis in this paper is very simple, so hopefully it can be
extended to analyze the convergence of numerical methods
for the complicated time-fractional evolution equations.

The observed convergence rate is better than that from the theoretical analysis when $\sigma_{m+1}< p+\alpha-1/2$.
Other techniques are needed in convergence analysis, which will be studied in our future work.

\appendix

\section{Proof of $c_n\geq 0$ for the  BN-$\theta$ method}\label{sec-C}
The  BN-$\theta$ method reduces to the FBDF-2 method for $\theta=0$ and
 to the GNGF-2 for $\theta = 1/2$. In this section,
we prove $c_n\geq 0$ for the  BN-$\theta$ method when $0\leq \theta \leq 1/2$.

Firstly, we give the proof of the following Lemma.\\
\textbf{Lemma A.1}
For $a^{(-\alpha)}_n=\frac{\Gamma(n+\alpha)}{\Gamma(\alpha)\Gamma(n+1)}$
and $0< \alpha < 1$, we have
\begin{eqnarray}
 \left(\frac{1+\alpha}{2}\right)^j a_{n-j}^{(-\alpha)}
&\leq& a_{n}^{(-\alpha)},\quad 0\leq j \leq n-1,  \label{app-C-2} \\
a_{n}^{(-\alpha-1)}-a_{n}^{(-\alpha)}
&\leq& (1+\alpha)\left(\frac{2+\alpha}{2}\right)^{n-2},\quad n\geq 0.\label{app-C-3}
\end{eqnarray}
\begin{proof}
From
$a^{(-\alpha)}_n=\frac{\Gamma(n+\alpha)}{\Gamma(\alpha)\Gamma(n+1)}$, we obtain
$a_{n}^{(-\alpha-1)}-a_{n}^{(-\alpha)}  =\frac{n}{\alpha}a_{n}^{(-\alpha)}$ and
\begin{eqnarray}
\frac{a^{(-\alpha)}_{n+1}}{a^{(-\alpha)}_n}&=&\frac{n+\alpha}{n+1}\geq \frac{1+\alpha}{2},
\quad n \geq 1,\label{appendix-C-1-2} \\
\frac{a_{n+1}^{(-\alpha-1)}-a_{n+1}^{(-\alpha)}}{a_{n}^{(-\alpha-1)}-a_{n}^{(-\alpha)}}
&=&\frac{n+\alpha}{n}\leq \frac{2+\alpha}{2},\quad n\geq 2.\label{appendix-C-1-3}
\end{eqnarray}
Eq. \eqref{appendix-C-1-2} implies
$ \frac{a^{(-\alpha)}_{n}}{a^{(-\alpha)}_{n-j}}
=\prod_{k=n-j}^{n-1}\frac{a^{(-\alpha)}_{k+1}}{a^{(-\alpha)}_k}
\geq \left(\frac{1+\alpha}{2}\right)^j$,
which completes the proof of \eqref{app-C-2}.
Obviously, \eqref{app-C-3} holds for $n=0,1$. From   \eqref{appendix-C-1-3},
we   obtain
$a_{n}^{(-\alpha-1)}-a_{n}^{(-\alpha)}
\leq \left(\frac{2+\alpha}{2}\right)^{n-2}(a_{2}^{(-\alpha-1)}-a_{2}^{(-\alpha)})
=(\alpha+1)\left(\frac{2+\alpha}{2}\right)^{n-2}$.
The proof complete.
\end{proof}

For $ 0\leq \theta \leq 1/2$, we have the following properties
\begin{eqnarray}
f_1(\theta)&=& \frac{\theta}{1+\theta\alpha}+ \frac{1-2\theta}{3-2\theta}
\leq f_1(({2+2\alpha})^{-1})= \frac{1+\alpha}{2+3\alpha},\label{APX-C-2} \\
f_2(\theta)&=&3\theta^2\left(\frac{1-2\theta}{3-2\theta}\right) +
\left(\frac{\theta}{1+\theta}\right)^3
+\left(\frac{1-2\theta}{3-2\theta}\right)^3 <\frac{8}{100},\label{APX-C-5}\\
f_3(\theta)&=&4\theta \left(\frac{1-2\theta}{3-2\theta}\right)^2+
\left(\frac{\theta}{1+\theta}\right)^3
+\left(\frac{1-2\theta}{3-2\theta}\right)^3<\frac{8}{100},\label{APX-C-6}\\
f_4(\theta)&=&{\theta} \frac{1-2\theta}{3-2\theta}+\frac{3}{8}\bigg[\frac{\theta^2}{(1+\theta)^2}
+\left(\frac{1-2\theta}{3-2\theta}\right)^2\bigg]<\frac{9}{100},\label{APX-C-7}\\
\rho_1&=&\frac{1-2\theta}{3-2\theta}\leq \frac{1}{3},\qquad
\rho_2=\frac{\alpha\theta}{1+\alpha\theta}\leq  \frac{\alpha}{2+\alpha},\label{APX-C-4}
\end{eqnarray}
where   we used
\begin{equation*}\begin{aligned}
&\max_{0\leq \theta \leq 1/2} f_1(\theta)= f_1(({2+2\alpha})^{-1}) ,\qquad\quad
\max_{0\leq \theta \leq 1/2} f_2(\theta)\approx f_2(0.3769)\approx 0.0685, \\
&\max_{0\leq \theta \leq 1/2} f_3(\theta)\approx f_3(0.1681)\approx 0.0602, \quad
\max_{0\leq \theta \leq 1/2} f_4(\theta)\approx f_4(0.2811)\approx 0.0806.
\end{aligned} \end{equation*}

\begin{proof}
From \eqref{eq:zeng-d} and \eqref{bntheta-b}, we have
$c_n= 2b_0{a}^{(-\alpha)}_n -\sum_{j=0}^{n}{b}_j{a}^{(-\alpha)}_{n-j},$
where $b_n$ is given by \eqref{b-n}.

\textbf{Step 1)} Prove $c_n\geq 0$ for $n\geq 3$. Let
\begin{equation}\label{rho}
\rho=\max\{\rho_1,\rho_2\},\qquad \rho_3= \rho({2+\alpha})/{2},\qquad
\lambda = ({1+\alpha})/{2}.
\end{equation}
By \eqref{b-n}, \eqref{rho},   \eqref{app-C-3}, and
$\sum_{j=1}^{n-1} a_{j}^{(-\alpha)}=a_{n}^{(-\alpha-1)}-1-a_n^{(-\alpha)}$,   we have
\begin{align}
b_n/b_0&=\rho_1\rho_2\sum_{j=1}^{n-1} a_{j}^{(-\alpha)} \rho_1^{j-1}\rho_2^{n-j-1}
+ \rho_2^n + \rho^n_1a_n^{(-\alpha)} \notag\\
&\leq  \rho_1\rho_2\rho^{n-2}\sum_{j=1}^{n-1} a_{j}^{(-\alpha)}+ \rho_2^n +  \rho^n_1 a_n^{(-\alpha)}\notag\\
&\leq  (1+\alpha) \rho_1\rho_2\rho_3^{n-2}  + \rho_2^n+\rho^n_1a_n^{(-\alpha)}.\label{appendix-C-3}
\end{align}
From  \eqref{app-C-2}, we have
$\lambda^{1-n}a_n^{(-\alpha)}/\alpha \geq1$ for $n\geq 1$.
Hence,
\begin{equation}\label{appendix-C-7} \begin{aligned}
b_n/b_0\leq&   \left[(1+\alpha) \rho_1\rho_2\rho_3^{n-2}+\rho_2^n\right]
\lambda^{1-n} {a_n^{(-\alpha)}}/{\alpha}
+\rho^n_1 a_n^{(-\alpha)}\\
=&\left[\frac{1+\alpha}{\alpha\lambda} \rho_1\rho_2  \left(\frac{\rho_3}{\lambda}\right)^{n-2}
+\frac{\lambda}{\alpha}\left(\frac{\rho_2}{\lambda}\right)^n+\rho^n_1\right]a_n^{(-\alpha)}\\
\leq&\bigg[\frac{1+\alpha}{\alpha\lambda} \rho_1\rho_2  \left(\frac{\rho_3}{\lambda}\right)
+\frac{\lambda}{\alpha}\left(\frac{\rho_2}{\lambda}\right)^3
+\rho^3_1   \bigg]a_n^{(-\alpha)},
\quad n\geq 3,
\end{aligned} \end{equation}
where we used  $\rho_1\leq 1/3$, $\rho_2/\lambda<1$,
and $\rho_3/\lambda<1$. Direct calculation yields
\begin{equation}\label{appendix-C-9-0}
\frac{1+\alpha}{\alpha} \frac{\rho_1\rho_2\rho_3}{\lambda^2}
=\left\{\begin{aligned}
& \frac{2\alpha(2+\alpha)}{ 1+\alpha} \frac{1-2\theta}{3-2\theta}\frac{\theta^2}{(1+\alpha\theta)^2}
\leq{3\theta^2} \frac{1-2\theta}{3-2\theta}  ,
& \rho_1\leq \rho_2 ,\\
&\frac{2(2+\alpha)}{1+\alpha} \left( \frac{1-2\theta}{3-2\theta}\right)^2 \frac{ \theta}{1+\alpha\theta}
\leq 4\theta \left(\frac{1-2\theta}{3-2\theta}\right)^2,
&\rho_1> \rho_2,
\end{aligned} \right.\end{equation}
\begin{equation}\label{appendix-C-9}
\frac{1}{\alpha} \frac{\rho_2^3}{\lambda^2} +\rho^3_1
=\frac{4\alpha^2}{ (1+\alpha)^2}\frac{\theta^3}{(1+\theta\alpha)^3}
+\left(\frac{1-2\theta}{3-2\theta}\right)^3
\leq  \left(\frac{\theta}{1+\theta}\right)^3
+\left(\frac{1-2\theta}{3-2\theta}\right)^3.
\end{equation}
Combining \eqref{appendix-C-7}, \eqref{appendix-C-9-0}, \eqref{appendix-C-9},
  \eqref{APX-C-5}, and \eqref{APX-C-6}  yields
\begin{equation}\label{appendix-C-11} \begin{aligned}
b_n/b_0\leq& \max\{f_2(\theta),f_3(\theta)\}a_n^{(-\alpha)}
 \leq \frac{2}{25}  a_n^{(-\alpha)},\quad n\geq 3.
\end{aligned} \end{equation}
From \eqref{APX-C-2}, we have
\begin{eqnarray}
{b_1}/{b_0}=\rho_2 + a_1^{(-\alpha)}\rho_1
=\alpha\left(\frac{\theta}{1+\theta\alpha}+ \frac{1-2\theta}{3-2\theta}\right)
\leq \frac{\alpha(1+\alpha)}{2+3\alpha}.\label{appendix-C-6-2}
\end{eqnarray}
Combining \eqref{appendix-C-11} and \eqref{appendix-C-6-2} yields
\begin{equation}\label{appendix-C-6-3} \begin{aligned}
\frac{b_0a_n^{(-\alpha)}+b_1a_{n-1}^{(-\alpha)}+ b_n}{b_0a_n^{(-\alpha)}}
\leq& \frac{27}{25}+\frac{3\alpha(1+\alpha)}{(2+\alpha)(2+3\alpha)}
\leq \frac{27}{25}+ \frac{2}{5}=\frac{37}{25} .
\end{aligned}\end{equation}
where we used ${a_{n-1}^{(-\alpha)}}
\leq \frac{n}{n-1+\alpha}{a_{n}^{(-\alpha)}}\leq \frac{3}{2+\alpha}{a_{n}^{(-\alpha)}}$
 for $n\geq 3$.

Using \eqref{APX-C-4}, we obtain
\begin{align*}
&  1+\alpha-2\rho_1  \geq 1+\alpha-2/3=({1+3\alpha})/{3},\\
& 1+\alpha-2\rho_2  \geq 1+\alpha- {2\alpha}/({2+\alpha})
= ({\alpha^2+\alpha+2})/({2+\alpha}),\\
& 1+\alpha-2\rho_3  \geq 1+\alpha-\rho({2+\alpha})
\geq 1+\alpha-({2+\alpha})/{3}  =({1+2\alpha})/{3},
\end{align*}
which leads to
\begin{align}
& \frac{\rho_1\rho_2}{1+\alpha-2\rho_3}+\frac{1}{1+\alpha}   \frac{\rho^2_2}{1+\alpha-2\rho_2}
+  \frac{\alpha}{2}\frac{\rho^2_1}{1+\alpha-2\rho_1} \notag\\
\leq &\frac{3\alpha}{1+2\alpha}\frac{\theta}{1+\theta\alpha}\frac{1-2\theta}{3-2\theta}
+\frac{\alpha(2+\alpha)}{(1+\alpha)(2+\alpha+\alpha^2)}
\frac{\alpha\theta^2}{(1+\alpha\theta)^2}\notag\\
&\qquad+\frac{3\alpha}{2(1+3\alpha)}\left(\frac{1-2\theta}{3-2\theta}\right)^2 \notag\\
\leq& {\theta} \frac{1-2\theta}{3-2\theta}+ \frac{3}{8}\left[\frac{\alpha\theta^2}{(1+\alpha\theta)^2}
+\left(\frac{1-2\theta}{3-2\theta}\right)^2\right] \notag\\
\leq& {\theta} \frac{1-2\theta}{3-2\theta}+\frac{3}{8}\left[\frac{\theta^2}{(1+\theta)^2}
+\left(\frac{1-2\theta}{3-2\theta}\right)^2\right]
< \frac{9}{100}.\qquad \text{(By \eqref{APX-C-7})} \label{appendix-C-14}
\end{align}
From    \eqref{appendix-C-3},
$a_j^{(-\alpha)}\leq a_2^{(-\alpha)}=\alpha(1+\alpha)/2$,  and
the following inequality,
\begin{equation*}\label{appendix-C-5} \begin{aligned}
\sum_{j=2}^{n-1}\rho_k^j a_{n-j}^{(-\alpha)}
=&\sum_{j=2}^{n-1}\frac{\rho_k^j}{\lambda^j} \left(\lambda^ja_{n-j}^{(-\alpha)}\right)
\leq a_{n}^{(-\alpha)} \sum_{j=2}^{n-1}\frac{\rho_k^j}{\lambda^j}
\leq a_{n}^{(-\alpha)} \frac{\rho_k^2/\lambda^2}{1-\rho_k/\lambda} \\
=&\frac{\rho^2_k}{\lambda(\lambda-\rho_k)} a_{n}^{(-\alpha)}
=\frac{4\rho^2_k}{(1+\alpha)(1+\alpha-2\rho_k)} a_{n}^{(-\alpha)},\quad k=1,2,3,
\end{aligned} \end{equation*}
we obtain
\begin{equation}\label{appendix-C-4} \begin{aligned}
\sum_{j=2}^{n-1}\frac{{b}_j}{b_0}{a}^{(-\alpha)}_{n-j}
\leq& \sum_{j=2}^{n-1} \left[(1+\alpha) \rho_1\rho_2\rho_3^{j-2}
+  \rho_2^j +  \frac{\alpha(1+\alpha)}{2}\rho^j_1\right]a_{n-j}^{(-\alpha)}\\
\leq&  {4}  \left(  \frac{\rho_1\rho_2}{1+\alpha-2\rho_3}
+\frac{1}{1+\alpha}\frac{\rho^2_2}{1+\alpha-2\rho_2}
+\frac{\alpha}{2}\frac{\rho^2_1}{1+\alpha-2\rho_1}\right)a_{n}^{(-\alpha)}\\
\leq & \frac{9}{25}a_{n}^{(-\alpha)}. \qquad \text{(By \eqref{appendix-C-14})}
\end{aligned} \end{equation}
Combining \eqref{appendix-C-6-3} and \eqref{appendix-C-4}    yields
\begin{equation*}\label{appendix-C-16}
b_0^{-1}\sum_{j=0}^{n}{b}_j{a}^{(-\alpha)}_{n-j}=\sum_{j=2}^{n-1}({b}_j/b_0){a}^{(-\alpha)}_{n-j}
+\left(b_0a_n^{(-\alpha)}+b_1a_{n-1}^{(-\alpha)}+ b_n\right)/b_0
\leq \frac{46}{25}a_n^{(-\alpha)},
\end{equation*}
which leads to
\begin{equation*} \label{appendix-C-20}\begin{aligned}
c_n=& 2b_0{a}^{(-\alpha)}_n  -b_0^{-1}\sum_{j=0}^{n}{b}_j{a}^{(-\alpha)}_{n-j}
\geq b_0\left(2 - \frac{46}{25}\right)a_n^{(-\alpha)}
=\frac{4}{25}b_0{a}^{(-\alpha)}_n \geq 0,\,n\geq 3.
\end{aligned}\end{equation*}

\textbf{Step 2)}
Prove $c_n> 0$ for $n=0,1,2$.  Obviously,
$c_0= 2b_0  - {b}_0\geq b_0 >0$ and
\begin{equation*}
c_1= (2-\alpha) b_0- b_1
\geq  \left(2 -\alpha -\frac{\alpha+\alpha^2}{2+3\alpha}\right)b_0
=\frac{4(1-\alpha^2) + 3\alpha}{2+3\alpha}b_0>0,
\end{equation*}
where we used \eqref{appendix-C-6-2}.
By \eqref{APX-C-2} and \eqref{APX-C-4}, we obtain
\begin{equation*} \label{appendix-C-d-n2}\begin{aligned}
b_2/b_0
=&\frac{\alpha^2\theta}{1+\theta\alpha}f_1(\theta)
+\frac{\alpha(1+\alpha)}{2}\left(\frac{1-2\theta}{3-2\theta}\right)^2
\leq  \frac{\alpha^2}{2+\alpha} \frac{1+\alpha}{2+3\alpha}
+\frac{\alpha(1+\alpha)}{2}\frac{1}{9}.
\end{aligned}\end{equation*}
From the above inequality   and \eqref{appendix-C-6-2},  we have
 \begin{align*}
c_2/b_0
=&  {a}^{(-\alpha)}_2
-\left(({b}_1/b_0)\alpha+{b}_2/b_0\right)\\
\geq&   \frac{\alpha(1+\alpha)}{2}-\left(\frac{\alpha^2(1+\alpha)}{2+3\alpha}
+\frac{\alpha^2}{2+\alpha} \frac{1+\alpha}{2+3\alpha}
+  \frac{\alpha(1+\alpha)}{2}  \frac{1}{9}\right) \\
= &{\alpha(1+\alpha)}\left(\frac{4}{9} - \frac{\alpha}{2+3\alpha}
-\frac{\alpha}{(2+\alpha)(2+3\alpha)}\right) \\
\geq&{\alpha(1+\alpha)}\left(\frac{4}{9} - \frac{1}{5}
-\frac{1}{15}\right)=\frac{8\alpha(1+\alpha)}{45}>0.
\end{align*}
The proof is complete.
\end{proof}

\section{Proofs of Lemmas \ref{lem4-2} and \ref{lem-10}}\label{sec-D}
Proof of Lemma \ref{lem4-2}.
\begin{proof}
For $\sigma\geq 0$, \eqref{eq:zeng-a2-3-2} follows   from
$\sum_{j=1}^{n-1}(n-j)^{-\alpha-1}j^{\sigma}
\leq n^{\sigma}\sum_{j=1}^{n-1}(n-j)^{-\alpha-1}
\lesssim  n^{\sigma} \sum_{j=1}^{\infty}j^{-\alpha-1} \lesssim n^{\sigma}$.
Next, we prove  \eqref{eq:zeng-a2-3-2} for $\sigma<0$.

For   $n\geq 2$, there exists   $j_n=\lceil n/2 \rceil$ and
$x_0=j_n/n \in (0,1)$ such that
\begin{align*}
\sum_{j=1}^{n-1}(n-j)^{-\alpha-1}j^{\sigma}
=&  \sum_{j=1}^{j_n}(n-j)^{-\alpha-1}j^{\sigma}
+ \sum_{j=j_n+1}^{n-1}(n-j)^{-\alpha-1}j^{\sigma}\\
\leq& \sum_{j=1}^{j_n}(n-j_n)^{-\alpha-1}j^{\sigma}
+ \sum_{j=j_n+1}^{n-1}(n-j)^{-\alpha-1}j_n^{\sigma}\\
\lesssim&  n^{-\alpha-1}\sum_{j=1}^{n-1}j^{\sigma}
+  n^{\sigma} \sum_{j=1}^{n-1}j^{-\alpha-1}.
\end{align*}
Using $\sum_{j=1}^{n-1}j^{-\alpha-1}\lesssim 1$ and
$\sum_{j=1}^{n-1}j^{\sigma} \lesssim n^{\sigma+1}\log(n)$
completes the proof of \eqref{eq:zeng-a2-3-2}.

By $0\leq a_{n}^{(-\alpha)}\lesssim n^{\alpha-1}$, one has
$$ \sum_{j=1}^{n}a_{n-j}^{(-\alpha)}j^{\sigma}\lesssim n^{\sigma}+ \sum_{j=1}^{n-1} {(n-j)}^{\alpha-1}j^{\sigma}.$$
Repeating the proof of  \eqref{eq:zeng-a2-3-2} finishes the proof of \eqref{eq:zeng-a2-3-3}.
The proof is completed.
\end{proof}


Proof of Lemma  \ref{lem-10}.
\begin{proof}
The condition  \eqref{cond-2} and Lemma \ref{lem2-1} yield the following linear system
\begin{equation*} \begin{aligned}
\sum_{j=1}^mw^{(m)}_{n,j}j^{\sigma_k}
&= \frac{\Gamma(\sigma_k+1)}{\Gamma(\sigma_k+1-\alpha)}n^{\sigma_k-\alpha}-\sum_{j=1}^n\omega_{n-j}^{(\alpha)}j^{\sigma_k}\\
&=O(n^{-\alpha-1}) + O(n^{\sigma_k-\alpha-p}), \qquad 1\leq k \leq m,
\end{aligned}\end{equation*}
which leads to
\begin{eqnarray}
&& |w^{(m)}_{n,k}|\lesssim n^{-\alpha-1} + n^{\sigma_m-p-\alpha},\quad 1\leq k \leq m.\label{assumption-e}
\end{eqnarray}
Combining \eqref{eq:Wnm}, Lemma \ref{lem4-2}, $\omega_n^{(\alpha)}=O(n^{-\alpha-1})$,
 \eqref{assumption-a}, and \eqref{assumption-e} leads to
\begin{equation}\label{eq:Wnk}
|W^{(m)}_{n,k}| \lesssim  n^{\max\{-\alpha-1,\sigma_m-p-\alpha\}},\quad 1\le k \le m.
\end{equation}
Combining \eqref{eq:zeng-a2-2}, \eqref{eq:zeng-a2-3-3}, and \eqref{eq:Wnk} yields \eqref{eq:Wnk-2},
which ends the proof.
\end{proof}

\section{Proof of Theorem \ref{thm6-1}}\label{sec-E}
\begin{proof}
We show a sketch of the proof.
Let $\theta^n={}_Fu_{h}^n-u_h^n$.
By  \eqref{s6-1}, \eqref{FLMM-3}, and $\theta^n=\varepsilon_{n}=0$ for $0\leq n \leq n_0-1$, we obtain
\begin{equation}\label{s6-3}\begin{aligned}
\frac{1}{\tau^{\alpha}}\sum_{j=n_0}^n\omega_{n-j}^{(\alpha)}(\theta^j,v)&
+(\nabla \theta^n,\nabla v)=\left(P_{h}\left(f({}_Fu^{n}_h)-f(u^{n}_h)\right),v\right)\\
&-\frac{1}{\tau^{\alpha}}\sum_{j=1}^{n-n_0}\varepsilon_{n-j}\omega_{n-j}^{(\alpha)}(\theta^j + u_h^j-u_h^0,v).
\end{aligned} \end{equation}

Similar to \eqref{e3.22}, we can obtain the equivalent form of \eqref{s6-3} as
\begin{equation}\label{s6-4}
(\mathcal{A}^{\alpha,n_0-1}_{\tau} \theta^n,v) + (\mathcal{B}^{\alpha,n_0-1}\nabla \theta^j,\nabla v)
=  \big( \mathcal{B}^{\alpha,n_0-1}\widetilde{F}^n ,v\big)
-\sum_{j=n_0}^{n} \widetilde{b}_{n-j}(\theta^j,v)-(H^n,v).
\end{equation}
where    $\widetilde{F}^n=f({}_Fu^{n}_h)-f(u^{n}_h))$, and
\begin{equation*}\begin{aligned}
\widetilde{b}_n= \frac{1}{\tau^{\alpha}}\sum_{j=n_0}^nb_{n-j}\varepsilon_{j}\omega_{j}^{(\alpha)},\qquad
H^n=\sum_{k=n_0}^{n}  {b}_{n-k}\sum_{j=1}^{k-n_0} \varepsilon_{k-j}\omega_{k-j}^{(\alpha)}(u_h^j-u_h^0).
\end{aligned} \end{equation*}
By $\omega_{n}^{(\alpha)}=O(n^{-\alpha-1})$ and \eqref{eq:zeng-a2-3-2},    we
can easily obtain
\begin{equation*}
|\widetilde{b}_n | \lesssim \varepsilon  n^{-\alpha-1}.
 \end{equation*}
By the boundedness of $\|u_h^n\|$, $b_n=O(n^{-\alpha-1})$, and $\omega_n^{(\alpha)}=O(n^{-\alpha-1})$, we  derive
\begin{equation*}
\|H^n\|\lesssim \sum_{k=n_0}^{n}  |{b}_{n-k}|\sum_{j=1}^{k} |\varepsilon_{k-j}\omega_{k-j}^{(\alpha)} |
\lesssim \varepsilon \sum_{j=n_0}^{n} |{b}_{n-k}|\lesssim  \varepsilon.
\end{equation*}
Following the proof of  Theorem \ref{thm2-1}, we can easily  arrive at \eqref{s6-2}, the details are omitted.
The proof is complete.
\end{proof}

\section*{Acknowledgment}
The authors are grateful to Professor Dongfang Li for his valuable comments on an earlier version of this paper. This work has been supported by the National Natural Science Foundation of China (12001326, 11771254), Natural Science Foundation of Shandong Province (ZR2019ZD42, ZR2020QA032), China
Postdoctoral Science Foundation (BX20190191, 2020M672038), the startup fund from Shandong
University (11140082063130). GEK would like to acknowledge support by the MURI/ARO on Fractional
PDEs for Conservation Laws and Beyond: Theory, Numerics and Applications (W911NF-15-1-0562)".

\bibliographystyle{ws-m3as}
\bibliography{fded19m1y17}

\end{document}